\documentclass[12pt]{amsart}  

\usepackage[latin1]{inputenc}
\usepackage{amsmath} 
\usepackage{amsfonts}
\usepackage{amssymb}
\usepackage{stmaryrd}
\usepackage{latexsym} 
\usepackage{graphicx}
\usepackage{subfigure}
\usepackage{color}
\usepackage{hyperref}
\usepackage{verbatim}
\usepackage[all]{xy}
\usepackage{graphics}
\usepackage{pdfsync}
\usepackage{xcolor}
\usepackage{tikz}
\usepackage{bm}
\usetikzlibrary{arrows.meta}
\usetikzlibrary{shapes.geometric}
\usetikzlibrary{calc,math}
\usetikzlibrary{external}
\usetikzlibrary{decorations.markings}

\theoremstyle{definition}
\newtheorem{theorem}{Theorem}[section]
\newtheorem{proposition}[theorem]{Proposition}

\newtheorem{corollary}[theorem]{Corollary}

\newtheorem{conjecture}[theorem]{Conjecture}
\newtheorem{definition}[theorem]{Definition}
\newtheorem{example}[theorem]{Example}

\newtheorem{definition/theorem}[theorem]{Definition/Theorem}

\theoremstyle{remark}
\newtheorem{remark}[theorem]{Remark}

\numberwithin{equation}{section}
\newlength\cellsize 
\savebox2{%
\begin{picture}(15,15)
\put(0,0){\line(1,0){15}}
\put(0,0){\line(0,1){15}}
\put(15,0){\line(0,1){15}}
\put(0,15){\line(1,0){15}}
\end{picture}}
\newcommand\cellify[1]{\def\thearg{#1}\def\nothing{}%
\ifx\thearg\nothing
\vrule width0pt height\cellsize depth0pt\else
\hbox to 0pt{\usebox2\hss}\fi%
\vbox to 15\unitlength{
\vss
\hbox to 15\unitlength{\hss$#1$\hss}
\vss}}
\newcommand\tableau[1]{\vtop{\let\\=\cr
\setlength\baselineskip{-16000pt}
\setlength\lineskiplimit{16000pt}
\setlength\lineskip{0pt}
\halign{&\cellify{##}\cr#1\crcr}}}
\savebox3{%
\begin{picture}(15,15)
\put(0,0){\line(1,0){15}}
\put(0,0){\line(0,1){15}}
\put(15,0){\line(0,1){15}}
\put(0,15){\line(1,0){15}}
\end{picture}}
\newcommand\expath[1]{%
\hbox to 0pt{\usebox3\hss}%
\vbox to 15\unitlength{
\vss
\hbox to 15\unitlength{\hss$#1$\hss}
\vss}}
\newcommand\bas[1]{\omit \vbox to \cellsize{ \vss \hbox to \cellsize{\hss$#1$\hss} \vss}}

\DeclareMathOperator{\FT}{FT}
\DeclareMathOperator\ftrip{\mathcal{F}}
\DeclareMathOperator{\sign}{sign}
\DeclareMathOperator{\sort}{sort}
\DeclareMathOperator\type{type}

\DeclareMathOperator{\ttrip}{\mathcal{T}}

\DeclareMathOperator{\ctrip}{\mathcal{C}}
\DeclareMathOperator{\strip}{\mathcal{S}}
\DeclareMathOperator{\ST}{ST}
\DeclareMathOperator{\trace}{trace}

\NewDocumentCommand{\vertexs}{O{black} O{1cm} O{above} m O{(0,0)} m O{}}{
\path let \p1 = #6, \p2=#5 in node[style={draw,fill,color=#1,circle,minimum size=6pt,inner sep=0},label={[text=#1, label distance=#2]#3:}] (v#7#4) at (\x1+\x2,\y1+\y2){};}
\NewDocumentCommand{\vertex}{O{black} O{-3pt} O{above} m O{(0,0)} m O{}}{
\path let \p1 = #6, \p2=#5 in node[style={draw,fill,color=#1,circle,minimum size=6pt,inner sep=0},label={[text=#1, label distance=#2]#3:#4}] (v#7#4) at (\x1+\x2,\y1+\y2){};}
\NewDocumentCommand{\vertexl}{O{black} O{0pt} O{0pt} m O{(0,0)} m O{}}{
\path let \p1 = #6, \p2=#5 in node[style={draw,fill,color=#1,circle,minimum size=6pt,inner sep=0},label={[text=#1, shift={(#2, #3)}]#4}] (v#7#4) at (\x1+\x2,\y1+\y2){};}

\usepackage[backend=bibtex]{biblatex}
\begin{document}

\title[The chromatic symmetric function of graphs glued at a single vertex]{The chromatic symmetric function of graphs glued at a single vertex}

\author{Foster Tom}
\address{Department of Mathematics, Massachusetts Institute of Technology, Cambridge, MA 02139}
\email{ftom@mit.edu}

\author{Aarush Vailaya}
\address{The Harker School, San Jose, CA 95124}
\email{vailaya.aarush@gmail.com}



\begin{abstract}
We describe how the chromatic symmetric function of two graphs glued at a single vertex can be expressed as a matrix multiplication using certain information of the two individual graphs. We then prove new $e$-positivity results by using a connection between forest triples, defined by the first author, and Hikita's probabilities associated to standard Young tableaux. Specifically, we prove that gluing a sequence of unit interval graphs and cycles results in an $e$-positive graph. We also prove $e$-positivity for a graph obtained by gluing the first and last vertices of such a sequence. This generalizes $e$-positivity of cycle-chord graphs and supports Ellzey's conjectured $e$-positivity for proper circular arc digraphs.
\end{abstract}


\maketitle
\section{Introduction}\label{section:introduction}

The Stanley--Stembridge conjecture \cite{chromsym, stanstem} is a monumental problem in algebraic combinatorics. It states that the chromatic symmetric function $X_G(\bm x)$ of a $(\bm 3+\bm 1)$-free graph $G$ is $e$-positive. By a reduction due to Guay-Paquet \cite{stanstemreduction}, it suffices to prove $e$-positivity whenever $G$ is a unit interval graph. Shareshian and Wachs \cite{chromposquasi} refined this conjecture by introducing the chromatic quasisymmetric function $X_G(\bm x;q)$, showing that it is symmetric whenever $G$ is a unit interval graph, and conjecturing that in this case, its $e$-coefficients are positive polynomials in $q$. Hikita \cite{stanstemproof} proved the Stanley--Stembridge conjecture by giving a combinatorial formula for $X_G(\bm x;q)$ in terms of probabilities associated to standard Young tableaux, whose summands are all nonnegative when $q=1$. The Shareshian--Wachs conjecture is still open and has seen keen interest because of its connection to Hessenberg varieties \cite{chromhesssplitting, dothessenberg, chromhessequi, hessenberghopf, stanstemhess, chromquasihessenberg} and Hecke algebras \cite{chromhecke, chromcharacters}. Ellzey \cite{chromquasidi} generalized further by defining a chromatic quasisymmetric function $X_{\vec G}(\bm x;q)$ for any directed graph $\vec G$, showing symmetry whenever $\vec G$ is a proper circular arc digraph, and conjecturing $e$-positivity in this case.\\

In this paper, we consider the operation of gluing graphs $G$ and $H$ together at a single vertex. In Section \ref{section:ftmatrix}, we show that the resulting chromatic symmetric function can be expressed as a matrix multiplication using certain information about $G$ and $H$. In Section \ref{section:tabmatrix}, we prove a surprising connection between forest triples, defined by the first author \cite{qforesttriples}, and Hikita's probabilities on standard Young tableaux (Theorem \ref{thm:samemats}). From this connection and Hikita's remarkable breakthrough, we prove $e$-positivity of graphs obtained from gluing a sequence of unit interval graphs and cycles (Corollary \ref{cor:uicycle}). In Section \ref{section:traceresult}, we prove that if $G^\circ$ is obtained by gluing two vertices of a graph $G$, then $X_{G^\circ}(\bm x)$ is equal to the trace of the matrix corresponding to $G$ (Theorem \ref{thm:trace:trace}). In particular, we prove $e$-positivity for graphs obtained by gluing the first and last vertices of a unit interval graph (Corollary \ref{cor:circulararcgraph}). These graphs are proper circular arc graphs, lending support to Ellzey's conjecture. This also generalizes $e$-positivity of cycle-chord graphs (Corollary~\ref{cor:noncrossingcyclechord}).
\section{Background}\label{section:background}

All graphs in this paper will have vertex set $[n]=\{1,\ldots,n\}$ for some $n$. For graphs $G=([n],E)$ and $H=([n'],E')$, we define the \emph{sum}
\begin{equation*}
G+H=([n+n'-1],\{E\cup \{\{i+n-1,j+n-1\}: \ \{i,j\}\in E'\}).
\end{equation*} 
Informally, we glue vertex $n$ of $G$ to vertex $1$ of $H$. The main focus of this paper is to study how the chromatic symmetric function behaves under this operation.

\begin{figure}
$$\begin{tikzpicture}
\node at (-0.75,0) {$G=$};
\draw (0,0)--(0.5,0.866)--(1.5,0.866)--(2,0)--(1.5,-0.866)--(0.5,-0.866)--(0,0)--(1.5,0.866)--(1.5,-0.866)--(0,0) (0.5,0.866)--(2,0)--(0.5,-0.866)--(0.5,0.866)--(1.5,-0.866) (0.5,-0.866)--(1.5,0.866);
\filldraw (0,0) circle (3pt) node[align=center,below] (1){1};
\filldraw (0.5,0.866) circle (3pt) node[align=center,above] (2){2};
\filldraw (0.5,-0.866) circle (3pt) node[align=center,below] (3){3};
\filldraw (1.5,0.866) circle (3pt) node[align=center,above] (4){4};
\filldraw (1.5,-0.866) circle (3pt) node[align=center,below] (5){5};
\filldraw [color=red] (2,0) circle (3pt) node[align=center,below] (6){6};
\node at (3.75,0) {$H=$};
\draw (4.5,0)--(5,0.866)--(5.5,0)--(6,0.866)--(6.5,0)--(4.5,0) (5,0.866)--(6,0.866);
\filldraw [color=red] (4.5,0) circle (3pt) node[align=center,below] (1){1};
\filldraw (5,0.866) circle (3pt) node[align=center,above] (2){2};
\filldraw (5.5,0) circle (3pt) node[align=center,below] (3){3};
\filldraw (6,0.866) circle (3pt) node[align=center,above] (4){4};
\filldraw (6.5,0) circle (3pt) node[align=center,below] (5){5};
\node at (8.75,0) {$G+H=$};
\draw (10,0)--(10.5,0.866)--(11.5,0.866)--(12,0)--(11.5,-0.866)--(10.5,-0.866)--(10,0)--(11.5,0.866)--(11.5,-0.866)--(10,0) (10.5,0.866)--(12,0)--(10.5,-0.866)--(10.5,0.866)--(11.5,-0.866) (10.5,-0.866)--(11.5,0.866) (12,0)--(12.5,0.866)--(13,0)--(13.5,0.866)--(14,0)--(12,0) (12.5,0.866)--(13.5,0.866);
\filldraw (10,0) circle (3pt) node[align=center,below] (1){1};
\filldraw (10.5,0.866) circle (3pt) node[align=center,above] (2){2};
\filldraw (10.5,-0.866) circle (3pt) node[align=center,below] (3){3};
\filldraw (11.5,0.866) circle (3pt) node[align=center,above] (4){4};
\filldraw (11.5,-0.866) circle (3pt) node[align=center,below] (5){5};
\filldraw [color=red] (12,0) circle (3pt) node[align=center,below] (6){6};
\filldraw (12.5,0.866) circle (3pt) node[align=center,above] (7){7};
\filldraw (13,0) circle (3pt) node[align=center,below] (8){8};
\filldraw (13.5,0.866) circle (3pt) node[align=center,above] (9){9};
\filldraw (14,0) circle (3pt) node[align=center,below] (10){10};
\end{tikzpicture}$$
\caption{\label{fig:glue} Graphs $G$ and $H$ glued at a single vertex.}
\end{figure}
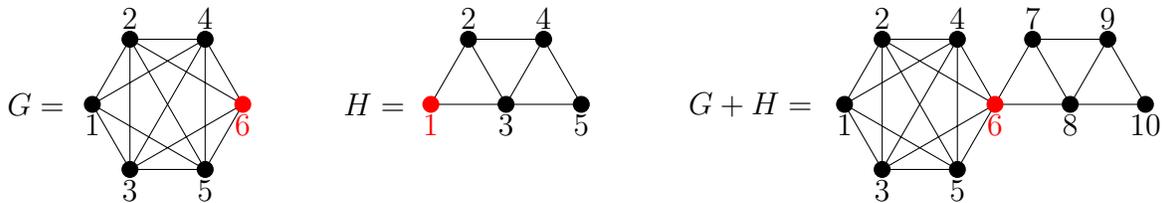

\subsection{Chromatic symmetric functions and unit interval graphs}
A \emph{proper colouring} of $G=([n],E)$ is a function $\kappa:[n]\to\mathbb P=\{1,2,3,\ldots\}$ such that $\kappa(u)\neq \kappa(v)$ whenever $\{u,v\}\in E$. The \emph{chromatic symmetric function} of $G$ is \cite[Definition 2.1]{chromsym}
\begin{equation*}
X_G(\bm x)=\sum_{\kappa\text{ proper colouring }}x_{\kappa(1)}x_{\kappa(2)}\cdots x_{\kappa(n)}.
\end{equation*}
We say that $X_G(\bm x)$ is \emph{$e$-positive} if $X_G(\bm x)$ has nonnegative coefficients when expanded in the basis of \emph{elementary symmetric functions} $e_\lambda$, defined by $e_\lambda=e_{\lambda_1}\cdots e_{\lambda_\ell}$, where $e_k=\sum_{i_1<\cdots<i_k}x_{i_1}\cdots x_{i_k}$. We say that $G$ is a \emph{natural unit interval graph (or NUIG)} if there are unit intervals $I_1=[a_1,a_1+1],\ldots,I_n=[a_n,a_n+1]\subset\mathbb R$ with $a_1<\cdots<a_n$ and $\{i,j\}\in E(G)$ if and only if $I_i\cap I_j\neq\emptyset$. Equivalently, for every $1\leq i<j<k\leq n$, if $\{i,k\}\in E(G)$, then $\{i,j\}\in E(G)$ and $\{j,k\}\in E(G)$. These are also equivalent to \emph{proper interval graphs}, where the intervals $I_i$ are allowed to be of different lengths but they cannot contain each other. Hikita proved the famous Stanley--Stembridge conjecture.

\begin{theorem}\cite[Conjecture 5.5]{chromsym}, \cite[Theorem 3]{stanstemproof}
If $G$ is an NUIG, then $X_G(\bm x)$ is $e$-positive.
\end{theorem}

An \emph{ascent} of a proper colouring $\kappa$ of $G=([n],E)$ is an edge $\{u,v\}\in E$ with $u<v$ and $\kappa(u)<\kappa(v)$. The \emph{chromatic quasisymmetric function} of $G$ is \cite[Definition 1.2]{chromposquasi}
\begin{equation*}
X_G(\bm x;q)=\sum_{\kappa\text{ proper colouring}}q^{\text{asc}(\kappa)}x_{\kappa(1)}x_{\kappa(2)}\cdots x_{\kappa(n)}.
\end{equation*}

Note that if we set $q=1$, we get $X_G(\bm x;1)=X_G(\bm x)$. Shareshian and Wachs showed that if $G$ is an NUIG, then $X_G(\bm x;q)$ is in fact symmetric \cite[Theorem 4.5]{chromposquasi}. They conjectured that in this case, $X_G(\bm x;q)$ is \emph{$e$-positive}, meaning that its $e$-coefficients, which are polynomials in the variable $q$, have positive coefficients. This refined conjecture is still open.

\begin{conjecture}
\cite[Conjecture 5.1]{chromposquasi}
If $G$ is an NUIG, then $X_G(\bm x;q)$ is $e$-positive. 
\end{conjecture}

\begin{figure}
\begin{tikzpicture}
\draw (1.5,0) node (1) {$G=$};
\filldraw (2.134,0.5) circle (3pt) node[align=center,above] (1){1};
\filldraw (2.134,-0.5) circle (3pt) node[align=center,below] (2){2};
\filldraw (3,0) circle (3pt) node[align=center,above] (3){3};
\filldraw (3.866,0.5) circle (3pt) node[align=center,above] (4){4};
\filldraw (3.866,-0.5) circle (3pt) node[align=center,below] (5){5};
\draw (2.134,0.5)--(2.134,-0.5)--(3,0)--(2.134,0.5) (3.866,0.5)--(3.866,-0.5)--(3,0)--(3.866,0.5);
\end{tikzpicture}
\begin{align*}
X_G(\bm x;q)&=(q^2+2q^3+q^4)e_{32}+(q+3q^2+4q^3+3q^4+q^5)e_{41}\\&\nonumber+(1+3q+4q^2+4q^3+4q^4+3q^5+q^6)e_5\\\nonumber&=q^2[2]_q[2]_qe_{32}+q[3]_q[2]_q[2]_qe_{41}+[5]_q[2]_q[2]_qe_5.
\end{align*}
\caption{\label{fig:chromsymexamplebowtie} The bowtie graph $G$ and the chromatic quasisymmetric function $X_G(\bm x;q)$.}
\end{figure}

\begin{example}
The complete graph $K_n$ is an NUIG and has $X_{K_n}(\bm x;q)=[n]_q!e_n$, where 
\begin{equation*}
[n]_q!=[n]_q[n-1]_q\cdots [2]_q[1]_q\text{ and }[k]_q=1+q+q^2+\cdots+q^{k-1}=\frac{q^k-1}{q-1}.
\end{equation*}
\end{example}

\begin{example}
The bowtie graph $G$ in Figure \ref{fig:chromsymexamplebowtie} is an NUIG and $X_G(\bm x;q)$ is $e$-positive.
\end{example}

\begin{example}
The path $P_n$ is an NUIG and has \cite[Section 5]{chromposquasi}
\begin{equation*}
X_{P_n}(\bm x;q)=\sum_{\alpha\vDash n}[\alpha_1]_q([\alpha_2]_q-1)\cdots([\alpha_\ell]_q-1)e_{\text{sort}(\alpha)},
\end{equation*}
so in particular, $X_{P_n}(\bm x;q)$ is $e$-positive.
\end{example}

More generally, for a directed graph $\vec G=([n],\vec E)$, Ellzey defined the \emph{chromatic quasisymmetric function} of $\vec G$ to be \cite[Definition 1.3]{chromquasidi}
\begin{equation*}
X_{\vec G}(\bm x;q)=\sum_{\kappa\text{ proper colouring}}q^{\text{asc}(\kappa)}x_{\kappa(1)}x_{\kappa(2)}\cdots x_{\kappa(n)},
\end{equation*}
where $\text{asc}(\kappa)$ is the number of directed edges $(u,v)\in \vec E$ with $\kappa(u)<\kappa(v)$. Note that if $G=([n],E)$ is a graph and we direct edges from the smaller vertex to the larger vertex, then we recover the earlier definition of Shareshian and Wachs. We say that $\vec G$ is a \emph{proper circular arc digraph} if there are arcs $A_1,\ldots,A_n$ on some circle with left endpoints $a_1,\ldots,a_n$ in clockwise order, that do not contain each other, such that $(i,j)\in \vec E$ if and only if $a_j\in A_i$. Ellzey showed that $X_{\vec G}(\bm x;q)$ is symmetric for such graphs \cite[Corollary 5.7]{chromquasidi} and conjectured $e$-positivity. This extension of the refined Stanley--Stembridge conjecture is open.

\begin{conjecture}\label{conj:circui}
\cite[Conjecture 1.4]{chromquasidi} If $\vec G$ is a proper circular arc digraph, then $X_{\vec G}(\bm x;q)$ is $e$-positive.
\end{conjecture}

\begin{figure}$$
\begin{tikzpicture}
\filldraw (0,0) circle (3pt) node[align=center,above] (1){1};
\filldraw (0.5,0.866) circle (3pt) node[align=center,above] (2){2};
\filldraw (1.5,0.866) circle (3pt) node[align=center,above] (3){3};
\filldraw (2,0) circle (3pt) node[align=center,above] (4){4};
\filldraw (1.5,-0.866) circle (3pt) node[align=center,below] (5){5};
\filldraw (0.5,-0.866) circle (3pt) node[align=center,below] (6){6};
\draw [-{Stealth[scale=1.5]}] (0,0)--(0.3,0.52);
\draw [-{Stealth[scale=1.5]}] (0.5,0.866)--(1.1,0.866);
\draw [-{Stealth[scale=1.5]}] (1.5,0.866)--(1.8,0.346);
\draw [-{Stealth[scale=1.5]}] (2,0)--(1.7,-0.52);
\draw [-{Stealth[scale=1.5]}] (1.5,-0.866)--(0.9,-0.866);
\draw [-{Stealth[scale=1.5]}] (0.5,-0.866)--(0.2,-0.346);
\draw (0,0)--(0.5,0.866)--(1.5,0.866)--(2,0)--(1.5,-0.866)--(0.5,-0.866)--(0,0);
\end{tikzpicture}$$
\begin{align*}
X_{\vec C_6}(\bm x;q)=& \ 2q^3e_{222}+4q^2[3]_qe_{42}+2q^2[3]_qe_{42}+3q^2[3]_q[2]_qe_{33}+6q[5]_qe_6\\\textcolor{red}{\alpha=}&\ \ \ \ \ \ \
\hspace{-0.065cm}\text{\tiny{\textcolor{red}{222}}} \ \ \ \ \ \ \ \ \ \ \ \ \ \ \hspace{-0.035cm}\text{\tiny{\textcolor{red}{42}}}\hspace{0.1cm} \ \ \ \ \ \ \ \ \ \ \ \ \
\text{\tiny{\textcolor{red}{24}}} \ \ \ \ \ \ \ \ \ \ \ \ \ \ \ \ \ \ \text{\tiny{\textcolor{red}{33}}} \ \ \ \ \ \ \ \ \ \ \ \hspace{0.2cm}
\text{\tiny{\textcolor{red}{6}}} \ \ \ \ \ \ \ \ \ \ \ \
\end{align*}
\caption{\label{fig:cycle} The directed cycle $\vec C_6$ and the chromatic quasisymmetric function $X_{\vec C_6}(\bm x;q)$.}
\end{figure}

\begin{example}
The directed cycle $\vec C_n$ is a proper circular arc digraph and \cite[Corollary 6.2]{chromquasidi}
\begin{equation}\label{eq:xqcn}
X_{\vec C_n}(\bm x;q)=\sum_{\alpha\vDash n}\alpha_1([\alpha_1]_q-1)([\alpha_2]_q-1)\cdots([\alpha_\ell]_q-1)e_{\text{sort}(\alpha)},
\end{equation}
so $X_{\vec C_n}(\bm x;q)$ is $e$-positive. The example of $\vec C_6$ is given in Figure \ref{fig:cycle}, along with the corresponding compositions. It turns out that if we change the directed edge $(n,1)$ to $(1,n)$, the resulting graph $\vec C_n'$ also has that $X_{\vec C_n'}(\bm x;q)$ is symmetric and $e$-positive \cite[Section 3]{smirnov}.
\end{example}

\subsection{Forest triples}

The first author proved a signed combinatorial formula for the $e$-expansion of $X_G(\bm x)$. In the case that $G$ is an NUIG, a similar formula is given for $X_G(\bm x;q)$. Let $G=([n],E)$ be a graph and fix a total ordering $\lessdot$ on $E$.  

\begin{definition}\label{def:background:nbc}
A \emph{no broken circuit (NBC) tree} of $G$ is a subtree $T$ whose edges do not contain a subset of the form $B=C\setminus\{\max(C)\}$, where $C$ is a set of edges that form a cycle and $\max(C)$ is the largest edge of $C$ under $\lessdot$.
\end{definition}

\begin{definition}
A \emph{tree triple of $G$} is an object $\mathcal T=(T,\alpha,r)$ consisting of the following data.
\begin{itemize}
\item $T$ is an NBC tree of $G$.
\item $\alpha$ is a composition of size $|T|$.
\item $r$ is an integer with $1\leq r\leq\alpha_1$, the first part of $\alpha$.
\end{itemize}
A \emph{forest triple of $G$} is a sequence $\mathcal F=(\mathcal T_1=(T_1,\alpha^{(1)},r_1),\ldots,\mathcal T_m=(T_m,\alpha^{(m)},r_m))$ of tree triples such that each vertex of $G$ is in exactly one tree and we have $\min(T_1)<\cdots<\min(T_m)$. Let $\text{FT}(G)$ denote the set of forest triples of $G$. The \emph{type} of $\mathcal F$ is the partition
\begin{equation*}
\text{type}(\mathcal F)=\text{sort}(\alpha^{(1)}\cdots\alpha^{(m)}),
\end{equation*}
given by concatenating the compositions and sorting to make a partition. The \emph{sign} of $\mathcal F$ is the integer
\begin{equation*}
\text{sign}(\mathcal F)=(-1)^{\sum_{i=1}^m(\ell(\alpha^{(i)})-1)}=(-1)^{\ell(\text{type}(\mathcal F))-m}.
\end{equation*}
We will often be interested in $\alpha^{(1)}_1$, the first part of the composition associated to the tree containing vertex $1$. We define the \emph{reduced type} of $\mathcal F$, denoted $\text{type}'(\mathcal F)$, to be the partition $\text{type}(\mathcal F)$ with an instance of $\alpha^{(1)}_1$ removed.
\end{definition}
Now we can use forest triples to calculate the chromatic symmetric function.

\begin{theorem}\cite[Theorem 5.10]{qforesttriples}
The chromatic symmetric function of $G$ satisfies
\begin{equation*}
X_G(\bm x)=\sum_{\mathcal F\in\text{FT}(G)}\text{sign}(\mathcal F)e_{\text{type}(\mathcal F)}.
\end{equation*}
\end{theorem}

In the case that $G$ is an NUIG, we will order edges lexicographically, so that if $\{i,j\}$ and $\{i',j'\}$ are edges of $G$ with $i<j$ and $i'<j'$, we define $\{i,j\}\lessdot\{i',j'\}$ if either $i<i'$, or $i=i'$ and $j<j'$. Then $T$ is an NBC tree of $G$ precisely when no vertex of $T$ has two larger neighbours. We also define a permutation of $V(T)$, denoted $\text{list}(T)$, as follows. We first read the smallest vertex of $T$. Then at each step, we read the smallest unread vertex of $T$ that is adjacent to a read vertex of $T$. Now for a forest triple $\mathcal F=(\mathcal T_i=(T_i,\alpha^{(i)},r_i))$ of $G$, we define the integer 
\begin{equation*}
\text{weight}(\mathcal F)=\text{inv}_G(\text{list}(T_1)\cdots\text{list}(T_m))+\sum_{i=1}^m(r_i-1),
\end{equation*}
where $\text{inv}_G(\sigma)=|\{(i,j): \ i<j, \ \sigma_i>\sigma_j, \{\sigma_j,\sigma_i\}\in E(G)\}|$ is the number of \emph{$G$-inversions} of the permutation $\sigma$. We can similarly use forest triples to calculate the chromatic quasisymmetric function.

\begin{figure}
\begin{tikzpicture}
\draw (3,0) node (1) {$G=$};
\filldraw (3.634,0.5) circle (3pt) node[align=center,above] (1){1};
\filldraw (3.634,-0.5) circle (3pt) node[align=center,below] (2){2};
\filldraw (4.5,0) circle (3pt) node[align=center,above] (3){3};
\filldraw (5.366,0.5) circle (3pt) node[align=center,above] (4){4};
\filldraw (5.366,-0.5) circle (3pt) node[align=center,below] (5){5};
\draw (3.634,0.5)--(3.634,-0.5)--(4.5,0)--(3.634,0.5) (5.366,0.5)--(5.366,-0.5)--(4.5,0)--(5.366,0.5);

\draw (8,2) node [color=teal] (){$\alpha=32$};
\draw (8,1.5) node [color=teal] (){$r=1,2,3$};
\draw (9.5,1.75) node [](){$\text{or}$};
\draw (11,2) node [color=teal] (){$\alpha=23$};
\draw (11,1.5) node [color=teal] (){$r=1,2$};
\draw (9.5,-4.5) node [](){\tiny{$-(\textcolor{orange}{1+2q+q^2})(\textcolor{teal}{1+q+q^2+1+q})$}};

\filldraw [color=teal](7.134,0.5) circle (3pt) node[align=center,above,color=teal] (1){1};
\filldraw [color=teal](7.134,-0.5) circle (3pt) node[align=center,below,color=teal] (2){2};
\filldraw [color=teal](8,0) circle (3pt) node[align=center,above,color=teal] (3){3};
\filldraw [color=teal](8.866,0.5) circle (3pt) node[align=center,above,color=teal] (4){4};
\filldraw [color=teal](8.866,-0.5) circle (3pt) node[align=center,below,color=teal] (5){5};
\draw [color=teal](7.134,0.5)--(7.134,-0.5)--(8,0) (8.866,0.5)--(8.866,-0.5) (8,0)--(8.866,0.5);
\draw (8,-1.25) node (){\small{$\text{inv}_G(\textcolor{teal}{12345})=\textcolor{orange}0$}};

\filldraw [color=teal](7.134,-2) circle (3pt) node[align=center,above,color=teal] (1){1};
\filldraw [color=teal](7.134,-3) circle (3pt) node[align=center,below,color=teal] (2){2};
\filldraw [color=teal](8,-2.5) circle (3pt) node[align=center,above,color=teal] (3){3};
\filldraw [color=teal](8.866,-2) circle (3pt) node[align=center,above,color=teal] (4){4};
\filldraw [color=teal](8.866,-3) circle (3pt) node[align=center,below,color=teal] (5){5};
\draw [color=teal](7.134,-2)--(7.134,-3)--(8,-2.5) (8.866,-2)--(8.866,-3) (8,-2.5)--(8.866,-3);
\draw (8,-3.75) node (){\small{$\text{inv}_G(\textcolor{teal}{12354})=\textcolor{orange}1$}};

\filldraw [color=teal](10.134,0.5) circle (3pt) node[align=center,above,color=teal] (1){1};
\filldraw [color=teal](10.134,-0.5) circle (3pt) node[align=center,below,color=teal] (2){2};
\filldraw [color=teal](11,0) circle (3pt) node[align=center,above,color=teal] (3){3};
\filldraw [color=teal](11.866,0.5) circle (3pt) node[align=center,above,color=teal] (4){4};
\filldraw [color=teal](11.866,-0.5) circle (3pt) node[align=center,below,color=teal] (5){5};
\draw [color=teal](10.134,0.5)--(11,0)--(10.134,-0.5) (11.866,0.5)--(11.866,-0.5) (11,0)--(11.866,0.5);
\draw (11,-1.25) node (){\small{$\text{inv}_G(\textcolor{teal}{13245})=\textcolor{orange}1$}};

\filldraw [color=teal](10.134,-2) circle (3pt) node[align=center,above,color=teal] (1){1};
\filldraw [color=teal](10.134,-3) circle (3pt) node[align=center,below,color=teal] (2){2};
\filldraw [color=teal](11,-2.5) circle (3pt) node[align=center,above,color=teal] (3){3};
\filldraw [color=teal](11.866,-2) circle (3pt) node[align=center,above,color=teal] (4){4};
\filldraw [color=teal](11.866,-3) circle (3pt) node[align=center,below,color=teal] (5){5};
\draw [color=teal] (10.134,-2)--(11,-2.5)--(10.134,-3) (11.866,-2)--(11.866,-3) (11,-2.5)--(11.866,-3);
\draw (11,-3.75) node (){\small{$\text{inv}_G(\textcolor{teal}{13254})=\textcolor{orange}2$}};

\draw (14.5,2) node [color=cyan] (){$\alpha=3$};
\draw (17.5,2) node [color=magenta] (){$\alpha'=2$};
\draw (16,1.75) node [](){$\text{and}$};
\draw (14.5,1.5) node [color=cyan] (){$r=1,2,3$};
\draw (17.5,1.5) node [color=magenta] (){$r'=1,2$};
\draw (16,-4.5) node [](){\tiny{$+(\textcolor{orange}{2+2q})(\textcolor{cyan}{1+q+q^2})(\textcolor{magenta}{1+q})$}};

\filldraw [color=cyan](13.634,0.5) circle (3pt) node[align=center,above,color=cyan] (1){1};
\filldraw [color=cyan](13.634,-0.5) circle (3pt) node[align=center,below,color=cyan] (2){2};
\filldraw [color=cyan](14.5,0) circle (3pt) node[align=center,above,color=cyan] (3){3};
\filldraw [color=magenta](15.366,0.5) circle (3pt) node[align=center,above,color=magenta] (4){4};
\filldraw [color=magenta](15.366,-0.5) circle (3pt) node[align=center,below,color=magenta] (5){5};
\draw [color=cyan](13.634,0.5)--(13.634,-0.5)--(14.5,0); \draw [color=magenta](15.366,0.5)--(15.366,-0.5);
\draw (14.5,-1.25) node (){\small{$\text{inv}_G(\textcolor{cyan}{123}\textcolor{magenta}{45})=\textcolor{orange}0$}};

\filldraw [color=cyan](13.634,-2) circle (3pt) node[align=center,above,color=cyan] (1){1};
\filldraw [color=cyan](13.634,-3) circle (3pt) node[align=center,below,color=cyan] (2){2};
\filldraw [color=cyan](14.5,-2.5) circle (3pt) node[align=center,above,color=cyan] (3){3};
\filldraw [color=magenta](15.366,-2) circle (3pt) node[align=center,above,color=magenta] (4){4};
\filldraw [color=magenta](15.366,-3) circle (3pt) node[align=center,below,color=magenta] (5){5};
\draw [color=cyan](13.634,-2)--(14.5,-2.5)--(13.634,-3); \draw [color=magenta](15.366,-2)--(15.366,-3);
\draw (14.5,-3.75) node (){\small{$\text{inv}_G(\textcolor{cyan}{132}\textcolor{magenta}{45})=\textcolor{orange}1$}};

\filldraw [color=magenta](16.634,0.5) circle (3pt) node[align=center,above,color=magenta] (1){1};
\filldraw [color=magenta](16.634,-0.5) circle (3pt) node[align=center,below,color=magenta] (2){2};
\filldraw [color=cyan](17.5,0) circle (3pt) node[align=center,above,color=cyan] (3){3};
\filldraw [color=cyan](18.366,0.5) circle (3pt) node[align=center,above,color=cyan] (4){4};
\filldraw [color=cyan](18.366,-0.5) circle (3pt) node[align=center,below,color=cyan] (5){5};
\draw [color=magenta](16.634,-0.5)--(16.634,0.5); 
\draw [color=cyan](17.5,0)--(18.366,0.5)--(18.366,-0.5);
\draw (17.5,-1.25) node (){\small{$\text{inv}_G(\textcolor{magenta}{12}\textcolor{cyan}{345})=\textcolor{orange}0$}};

\filldraw [color=magenta](16.634,-2) circle (3pt) node[align=center,above,color=magenta] (1){1};
\filldraw [color=magenta](16.634,-3) circle (3pt) node[align=center,below,color=magenta] (2){2};
\filldraw [color=cyan](17.5,-2.5) circle (3pt) node[align=center,above,color=cyan] (3){3};
\filldraw [color=cyan](18.366,-2) circle (3pt) node[align=center,above,color=cyan] (4){4};
\filldraw [color=cyan](18.366,-3) circle (3pt) node[align=center,below,color=cyan] (5){5};
\draw [color=magenta](16.634,-3)--(16.634,-2); 
\draw [color=cyan](17.5,-2.5)--(18.366,-3)--(18.366,-2);
\draw (17.5,-3.75) node (){\small{$\text{inv}_G(\textcolor{magenta}{12}\textcolor{cyan}{354})=\textcolor{orange}1$}};
\end{tikzpicture}
\caption{\label{fig:ftexample} The bowtie graph $G$ and the forest triples of $G$ of type $32$.}
\end{figure}
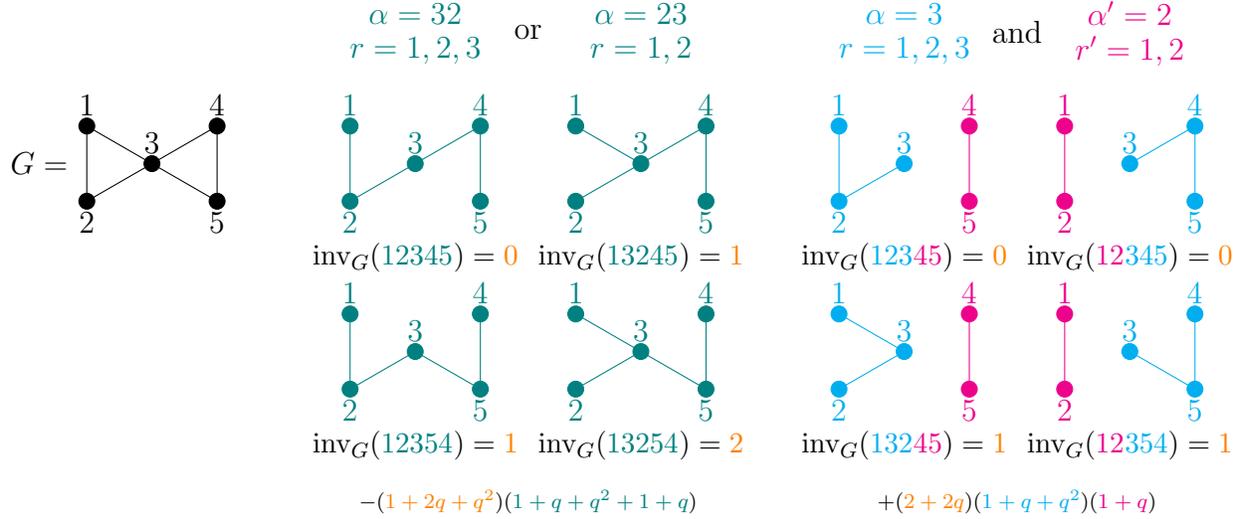

\begin{theorem}\cite[Theorem 3.4]{qforesttriples} The chromatic quasisymmetric function of an NUIG $G$ satisfies
\begin{equation*}
X_G(\bm x;q)=\sum_{\mathcal F\in\text{FT}(G)}\text{sign}(\mathcal F)q^{\text{weight}(\mathcal F)}e_{\text{type}(\mathcal F)}.
\end{equation*}
\end{theorem}

\begin{example}
Figure \ref{fig:ftexample} shows the forest triples of type $32$ for the bowtie graph $G$ from Figure \ref{fig:chromsymexamplebowtie}. We could have a single tree triple $\mathcal T=(T,\alpha,r)$, where either $\alpha=32$ and $1\leq r\leq 3$, or $\alpha=23$ and $1\leq r\leq 2$. Or we could have two tree triples of the form $\mathcal T=(T,3,r)$ and $\mathcal T'=(T',2,r')$, where $1\leq r\leq 3$ and $1\leq r'\leq 2$. By calculating the weights, we see that the coefficient of $e_{32}$ in $X_G(\bm x;q)$ is
\begin{equation*}
-(1+2q+q^2)(1+q+q^2+1+q)+(2+2q)(1+q+q^2)(1+q)=q^2(1+q)^2=q^2[2]_q[2]_q.
\end{equation*}
\end{example}
\subsection{Standard Young tableaux}

Hikita proved a formula for the $e$-expansion of $X_G(\bm x;q)$ in terms of probabilities associated to standard Young tableaux.

\begin{definition}
The \emph{Young diagram} of a partition $\lambda$ is the left-justified array of boxes with $\lambda_i$ boxes in the $i$-th row (from the bottom). A \emph{standard Young tableau (SYT) of shape $\lambda$} is a filling $T$ of these boxes with the integers $1$ through $n$, using each exactly once, so that rows and columns are increasing. We let $\text{SYT}_\lambda$ denote the set of SYT of shape $\lambda$ and we let $\text{SYT}_n=\bigsqcup_{\lambda\vdash n}\text{SYT}_\lambda$. If $T\in\text{SYT}_n$, we write $\text{shape}(T)=\lambda$.
\end{definition}

We can imagine constructing $T$ by successively dropping boxes labelled $1,\ldots,n$ on the top of some column. If box $k$ is dropped in column $b_k$, then the condition that $T$ has increasing rows is precisely that the sequence $b(T)=(b_1,\ldots,b_n)$ is a \emph{ballot sequence}, meaning that every initial string $(b_1,\ldots,b_j)$ contains at least as many $i$'s as $(i+1)$'s for every $i$. \\

We will often be interested in $b_n(T)$, the column of $T$ that contains the largest entry $n$, or in other words, the length of the row of $T$ that contains $n$. We define the \emph{reduced shape} of $T$, denoted $\text{shape}'(T)$, to be the partition $\text{shape}(T)$ with an instance of $b_n(T)$ removed.\\

Now given an NUIG $G$, Hikita described a certain subset of standard Young tableaux $\text{SYT}(G)\subseteq\text{SYT}_n$. In his formulation, he reads the vertices of $G$ in reverse order, from $n$ to $1$, so we denote by $\bar G$ the reverse graph
\begin{equation*}
\bar G=([n],\{\{n+1-j,n+1-i\}: \ \{i,j\}\in E(G)\}).
\end{equation*}
In other words, $i$ is a $\bar G$-neighbour of $j$ if and only if $(n+1-i)$ is a $G$-neighbour of $(n+1-j)$.

\begin{definition}
Let $G$ be an NUIG, let $T\in\text{SYT}_n$, and let $1\leq k\leq n$. Let $T\vert_{<k}$ the tableau obtained by restricting $T$ to the entries less than $k$ and let $\text{cols}_k$ be the number of columns of $T\vert_{<k}$. We now define the sequence $\delta^{(k)}(T)=(\delta^{(k)}_0,\delta^{(k)}_1,\ldots,\delta^{(k)}_{\text{cols}_k},\delta^{(k)}_{\text{cols}_k+1})$ by $\delta^{(k)}_0=1$, $\delta^{(k)}_{\text{cols}_k+1}=0$, and
\begin{equation*}
\delta^{(k)}_i=\begin{cases}1&\text{ if the top entry of column }i\text{ of }T\vert_{<k}\text{ is a }\bar G\text{-neighbour of }k,\\0&\text{ otherwise,}
\end{cases}
\end{equation*}
for $1\leq i\leq\text{cols}_k$. We define the sets of column indices
\begin{align*}
R_k(T)&=\{1\leq i\leq\text{cols}_k: \ \delta^{(k)}_i=1, \ \delta^{(k)}_{i-1}=0\}\text{ and }\\W_k(T)&=\{1\leq i\leq\text{cols}_k+1: \ \delta^{(k)}_i=0, \ \delta^{(k)}_{i-1}=1\}.
\end{align*}
Finally, we define $\text{SYT}(G)$ to be the subset of those $T\in\text{SYT}_n$ such that for every $1\leq k\leq n$, we have $b_k(T)\in W_k(T)$, where $b_k(T)$ denotes the column that contains the entry $k$. 
\end{definition}

Informally, a standard Young tableau $T\in\text{SYT}(G)$ is constructed as follows. We start with the empty tableau. Then for $1\leq k\leq n$, after we have dropped $(k-1)$ boxes, some columns of our partial tableau $T\vert_{<k}$ will have a $\bar G$-neighbour of $k$ as the top cell and some will not. We have $\delta^{(k)}_i=1$ if $i$ is a ``neighbour column'' and $\delta^{(k)}_i=0$ if $i$ is a ``non-neighbour column''. Then $R_k(T)$ is the set of the ``leftmost neighbour columns'' in a string and $W_k(T)$ is the set of the ``leftmost non-neighbour columns'' in a string. The condition that $b_k(T)\in W_k(T)$ is that the box $k$ must be dropped in a ``leftmost non-neighbour column''.\\

Note that if $i\in W_k(T)$, this means that the top entry of column $(i-1)$ of $T\vert_{<k}$ is a $\bar G$-neighbour of $k$ and the top entry of column $i$ is a $\bar G$-non-neighbour of $k$. Therefore, column $(i-1)$ received a box more recently than column $i$, so by induction, the $(i-1)$-th column was strictly taller than the $i$-th column and we can indeed drop box $k$ in column $i$. \\

We now define factors that we will associate to a standard Young tableau $T\in\text{SYT}(G)$. Hikita's original formula was slightly different but Guay-Paquet showed that this expression is equivalent.

\begin{definition} Let $G$ be an NUIG and let $T\in\text{SYT}(G)$. We define the rational function
\begin{equation*}
c_T(G)=\prod_{k=1}^nc^{(k)}_T(G),\text{ where }
c^{(k)}_T(G)=q^{|\{j>b_k: \ \delta_j=1\}|}[b_k]_q\frac{\prod_{i\in R_k(T)}[|i-b_k|]_q}{\prod_{i\in W_k(T): \ i\neq b_k}[|i-b_k|]_q}.
\end{equation*}
\end{definition}

Now Hikita proved the following formula for $X_G(\bm x;q)$.
\begin{theorem}\cite[Theorem 3]{stanstemproof} If $G$ is an NUIG, then we have
\begin{equation*}
X_G(\bm x;q)=\sum_{T\in\text{SYT}(G)}c_T(G)e_{\text{shape}(T)}.
\end{equation*}
In particular, if we evaluate at $q=1$, we have that $X_G(\bm x)$ is $e$-positive.
\end{theorem}

\begin{remark}
We can think of the factor \begin{equation*}p(b_k,q)=\frac{\prod_{i\in R_k(T)}[|i-b_k|]_q}{\prod_{i\in W_k(T): \ i\neq b_k}[|i-b_k|]_q}\end{equation*} as the \emph{probability} that box $k$ is dropped in column $b_k$ because for every $q\in\mathbb R_{\geq 0}$, we have $0\leq p(b_k,q)\leq 1$ and $\sum_{b_k\in W_k(T)}p(b_k,q)=1$.  
\end{remark}

\begin{proposition}\label{prop:tabfirstrow}
Let $G$ be an NUIG with $\{n-k+1,n\}\in E(G)$, in other words, the vertices $\{1,\ldots,k\}$ are all adjacent in $\bar G$. Then every tableau $T\in\text{SYT}(G)$ must have all of the entries $\{1,\ldots,k\}$ in the first row. 
\end{proposition}

\begin{proof}
Because the vertices $\{1,\ldots,k\}$ are all adjacent in $G$, we have for every $1\leq k'\leq k$ that $\delta^{(k')}(T)=(1,\ldots,1,0)$, so the next entry $k'$ must also be placed in the first row. 
\end{proof}

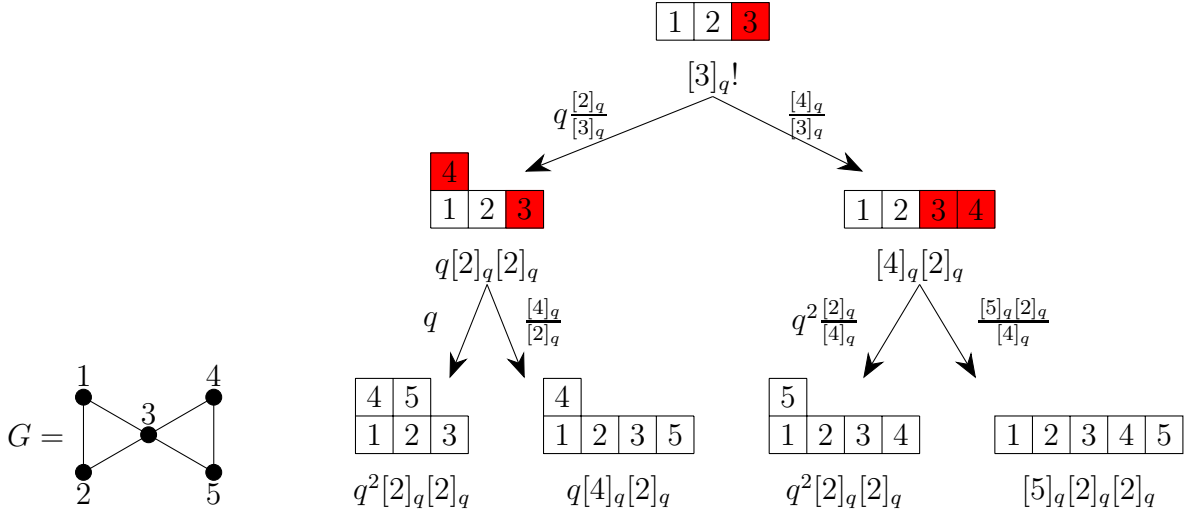
\begin{figure}
$$
\begin{tikzpicture}
\filldraw (11.75,5.25)--(12.25,5.25)--(12.25,5.75)--(11.75,5.75)--(11.75,5.25) [red];
\draw (11,5.5) node (){1};
\draw (11.5,5.5) node (){2};
\draw (12,5.5) node (){3};
\draw (11.5,4.75) node (){$[3]_q!$};
\draw (10.75,5.25)--(12.25,5.25) (10.75,5.75)--(12.25,5.75) (10.75,5.25)--(10.75,5.75) (11.25,5.25)--(11.25,5.75) (11.75,5.25)--(11.75,5.75) (12.25,5.25)--(12.25,5.75);

\draw [arrows = {-Stealth[scale=2]}] (11.5,4.5)--(9,3.5);
\draw (9.75,4.25) node () {$\tiny{q\frac{[2]_q}{[3]_q}}$};
\draw [arrows = {-Stealth[scale=2]}] (11.5,4.5)--(13.5,3.5);
\draw (12.75,4.25) node (){$\tiny{\frac{[4]_q}{[3]_q}}$};

\filldraw (7.75,3.25)--(8.25,3.25)--(8.25,3.75)--(7.75,3.75)--(7.75,3.25) [red];
\filldraw (8.75,2.75)--(9.25,2.75)--(9.25,3.25)--(8.75,3.25)--(8.75,2.75) [red];
\draw (8,3.5) node (){4};
\draw (8,3) node (){1};
\draw (8.5,3) node (){2};
\draw (9,3) node (){3};
\draw (8.5,2.25) node (){$q[2]_q[2]_q$};
\draw (7.75,2.75)--(9.25,2.75) (7.75,3.25)--(9.25,3.25) (7.75,3.75)--(8.25,3.75) (7.75,2.75)--(7.75,3.75) (8.25,2.75)--(8.25,3.75) (8.75,2.75)--(8.75,3.25) (9.25,2.75)--(9.25,3.25);

\draw [arrows = {-Stealth[scale=2]}] (8.5,2)--(8,0.75);
\draw (7.75,1.5) node (){$\tiny{q}$};
\draw [arrows = {-Stealth[scale=2]}] (8.5,2)--(9,0.75);
\draw (9.25,1.5) node (){$\tiny{\frac{[4]_q}{[2]_q}}$};

\filldraw (14.25,2.75)--(14.75,2.75)--(14.75,3.25)--(14.25,3.25)--(14.25,2.75) [red];
\filldraw (14.75,2.75)--(15.25,2.75)--(15.25,3.25)--(14.75,3.25)--(14.75,2.75) [red];
\draw (13.5,3) node (){1};
\draw (14,3) node (){2};
\draw (14.5,3) node (){3};
\draw (15,3) node (){4};
\draw (14.25,2.25) node (){$[4]_q[2]_q$};
\draw (13.25,2.75)--(15.25,2.75) (13.25,3.25)--(15.25,3.25) (13.25,2.75)--(13.25,3.25) (13.75,2.75)--(13.75,3.25) (14.25,2.75)--(14.25,3.25) (14.75,2.75)--(14.75,3.25) (15.25,2.75)--(15.25,3.25);

\draw [arrows = {-Stealth[scale=2]}] (14.25,2)--(13.5,0.75);
\draw (13,1.5) node (){$\tiny{q^2\frac{[2]_q}{[4]_q}}$};
\draw [arrows = {-Stealth[scale=2]}] (14.25,2)--(15,0.75);
\draw (15.5,1.5) node (){$\tiny{\frac{[5]_q[2]_q}{[4]_q}}$};

\draw (2.5,0) node (1) {$G=$};
\filldraw (3.134,0.5) circle (3pt) node[align=center,above] (1){1};
\filldraw (3.134,-0.5) circle (3pt) node[align=center,below] (2){2};
\filldraw (4,0) circle (3pt) node[align=center,above] (3){3};
\filldraw (4.866,0.5) circle (3pt) node[align=center,above] (4){4};
\filldraw (4.866,-0.5) circle (3pt) node[align=center,below] (5){5};
\draw (3.134,0.5)--(3.134,-0.5)--(4,0)--(3.134,0.5) (4.866,0.5)--(4.866,-0.5)--(4,0)--(4.866,0.5);
\draw (7,0.5) node (){4};
\draw (7.5,0.5) node (){5};
\draw (7,0) node (){1};
\draw (7.5,0) node (){2};
\draw (8,0) node (){3};
\draw (7.5,-0.75) node (){$q^2[2]_q[2]_q$};
\draw (6.75,-0.25)--(8.25,-0.25) (6.75,0.25)--(8.25,0.25) (6.75,0.75)--(7.75,0.75) (6.75,-0.25)--(6.75,0.75) (7.25,-0.25)--(7.25,0.75) (7.75,-0.25)--(7.75,0.75) (8.25,-0.25)--(8.25,0.25);
\draw (9.5,0.5) node (){4};
\draw (11,0) node (){5};
\draw (9.5,0) node (){1};
\draw (10,0) node (){2};
\draw (10.5,0) node (){3};
\draw (10.25,-0.75) node (){$q[4]_q[2]_q$};
\draw (9.25,-0.25)--(11.25,-0.25) (9.25,0.25)--(11.25,0.25) (9.25,0.75)--(9.75,0.75) (9.25,-0.25)--(9.25,0.75) (9.75,-0.25)--(9.75,0.75) (10.25,-0.25)--(10.25,0.25) (10.75,-0.25)--(10.75,0.25) (11.25,-0.25)--(11.25,0.25);
\draw (12.5,0.5) node (){5};
\draw (14,0) node (){4};
\draw (12.5,0) node (){1};
\draw (13,0) node (){2};
\draw (13.5,0) node (){3};
\draw (13.25,-0.75) node (){$q^2[2]_q[2]_q$};
\draw (12.25,-0.25)--(14.25,-0.25) (12.25,0.25)--(14.25,0.25) (12.25,0.75)--(12.75,0.75) (12.25,-0.25)--(12.25,0.75) (12.75,-0.25)--(12.75,0.75) (13.25,-0.25)--(13.25,0.25) (13.75,-0.25)--(13.75,0.25) (14.25,-0.25)--(14.25,0.25);
\draw (17,0) node (){4};
\draw (17.5,0) node (){5};
\draw (15.5,0) node (){1};
\draw (16,0) node (){2};
\draw (16.5,0) node (){3};
\draw (16.5,-0.75) node (){$[5]_q[2]_q[2]_q$};
\draw (15.25,-0.25)--(17.75,-0.25)--(17.75,0.25)--(15.25,0.25)--(15.25,-0.25) (17.25,-0.25)--(17.25,0.25) (15.75,-0.25)--(15.75,0.25) (16.25,-0.25)--(16.25,0.25) (16.75,-0.25)--(16.75,0.25);
\end{tikzpicture}$$

\caption{\label{fig:tabexample} The bowtie graph $G$ and the standard Young tableaux in $\text{SYT}(G)$.}
\end{figure}

\begin{example}
Figure \ref{fig:tabexample} shows the tableaux $T\in\text{SYT}(G)$ for the bowtie graph $G$. By Proposition \ref{prop:tabfirstrow}, because $\{3,5\}\in E(G)$, the entries $1,2,3$ must be in the first row. Now for $k=4$, the only $\bar G$-neighbour of $4$ is $3$, which we have coloured in red, so $\delta^{(4)}=(1,0,0,1,0)$. Note that we index $\delta^{(4)}$ starting at $0$. Then, $R_4(T)=\{3\}$ and $W_4(T)=\{1,4\}$, so the box $4$ can be dropped in the first or the fourth column. If $b_4(T)=1$, then $c^{(4)}_T(G)=q\frac{[2]_q}{[3]_q}$, and if $b_4(T)=4$, then $c^{(4)}_T(G)=\frac{[4]_q}{[3]_q}$. 
Next, when $k=5$, the $\bar G$-neighbours are $3$ and $4$, and we have given the corresponding factors $c^{(5)}_T(G)$ and the final coefficients $c_T(G)$. We have
\begin{equation*}
X_G(\bm x;q)=q^2[2]_q[2]_qe_{32}+(q[4]_q[2]_q+q^2[2]_q[2]_q)e_{41}+[5]_q[2]_q[2]_qe_5.
\end{equation*} \end{example}
\section{Forest triple matrix}\label{section:ftmatrix}
In this section, we use forest triples to define a matrix that will allow us to calculate a chromatic symmetric function $X_{G+H}(\bm x)$ in terms of information about $G$ and $H$. Informally, the idea is that a forest triple of $G+H$ essentially consists of a forest triple of $G+P_j$ for some $j$, along with a forest triple of $H$.

\begin{definition}\label{def:fmat:fti}
Let $G=([n],E)$ be a graph and consider the subset of $\text{FT}(G+P_j)$
\begin{equation*}
\text{FT}^{(i)}(G+P_j)=\{\mathcal F\in\text{FT}(G+P_j): \ \alpha_1=i, \ r=1, \ \{n,n+1,\ldots,n+j-1\}\subseteq T', \ \alpha'_\ell\geq j\},
\end{equation*}
where $\mathcal T=(T,\alpha,r)$ is the tree triple of $\mathcal F$ with $1\in T$ and $\mathcal T'=(T',\alpha',r')$ is the tree triple of $\mathcal F$ with $n\in T'$. In other words, the tree containing vertex $1$ must have a composition with first part $i$, the tree containing vertex $n$ must contain the entire path $P_j$, and its composition must have last part at least $j$. Note that we could have $\mathcal T=\mathcal T'$. \end{definition}
\begin{definition}\label{def:ftm:Fgij} We now define the infinite matrix $F_G$ by
\begin{equation*}
(F_G)_{i,j}=\sum_{\mathcal F\in\text{FT}^{(i)}(G+P_j)}\text{sign}(\mathcal F)e_{\text{type}'(\mathcal F)}.
\end{equation*}
If $G$ is an NUIG, we also define the infinite matrix $F_G(q)$ by
\begin{equation*}
(F_G(q))_{i,j}=\sum_{\mathcal F\in\text{FT}^{(i)}(G+P_j)}\text{sign}(\mathcal F)q^{\text{weight}(\mathcal F)}e_{\text{type}'(\mathcal F)}.
\end{equation*}
\end{definition}
We first calculate some examples.

\begin{proposition}\label{prop:ftp1}
    For the graph $P_1$ with one vertex, $F_{P_1}(q)$ is the infinite identity matrix $I$.
\end{proposition}
\begin{proof}
    We are considering forest triples of $P_1+P_j=P_j$. Since $n=1$, then $\ttrip=\ttrip'$, meaning there is only 1 forest triple, which consists of the single tree triple $\mathcal T=(P_j,j,1)$.
\end{proof}

\begin{proposition}\label{prop:ftp2}
For the two-vertex path $P_2$, we have
\begin{equation*}
(F_{P_2}(q))_{i,j}=\begin{cases} q[j-1]_qe_j&\text{ if }i=1,\\
1&\text{ if }i=j+1,\\
0&\text{ otherwise}.\end{cases}
\hspace{40pt}
F_{P_2}(q)=\left[\begin{matrix}
0&qe_2&q[2]_qe_3&q[3]_qe_4&\cdots\\
1&0&0&0&\cdots\\
0&1&0&0&\cdots\\
0&0&1&0&\cdots\\
0&0&0&1&\cdots\\
\vdots&\vdots&\vdots&\vdots&\ddots\\
\end{matrix}\right].
\end{equation*}
\end{proposition}

\begin{proof}
We are considering forest triples of $P_2+P_j=P_{j+1}$. If the vertices $1$ and $2$ are in different tree triples $\mathcal T=(T,\alpha,r)$ and $\mathcal T'=(T',\alpha',r')$, then we must have $\alpha=1$ and $r=1$. Because $\alpha'_\ell\geq j$, we must have $\alpha'=j$ and there are $j$ possibilities corresponding to the choices of $1\leq r'\leq j$. If the vertices $1$ and $2$ are in the same tree triple $\mathcal T=(T,\alpha,r)$, then $T$ must be the entire path $P_{j+1}$. Because $\alpha_\ell\geq j$, the only possibilities are $\alpha=(j+1)$ and $\alpha=1\ j$, and we must have $r=1$. These forest triples are given in Figure \ref{fig:ftp2example}.
\end{proof}
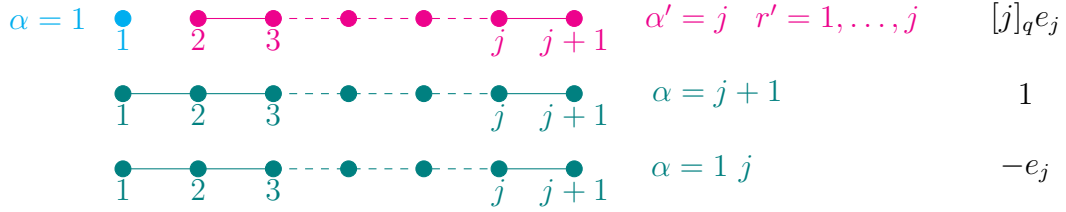
\begin{figure}
$$\begin{tikzpicture}
\node at (7.9,-1) [color=teal] {$\alpha=j+1$};
\node at (7.7,-2) [color=teal] {$\alpha=1 \ j$};
\node at (7.5,0) [color=magenta] {$\alpha'=j$};
\node at (9.5,0) [color=magenta] {$r'=1,\ldots,j$};
\node at (-1,0) [color=cyan] {$\alpha=1$};
\node at (12,-1) {$1$};
\node at (12,-2) {$-e_j$};
\node at (12,0) {$[j]_qe_j$};
\filldraw [color=teal](0,-1) circle (3pt) node[align=center,below,color=teal] (1){1};
\filldraw [color=teal](1,-1) circle (3pt) node[align=center,below,color=teal] (2){2};
\filldraw [color=teal](2,-1) circle (3pt) node[align=center,below,color=teal] (3){3};
\filldraw [color=teal](3,-1) circle (3pt) node[align=center,below,color=teal] (){};
\filldraw [color=teal](4,-1) circle (3pt) node[align=center,below,color=teal] (){};
\filldraw [color=teal](5,-1) circle (3pt) node[align=center,below,color=teal] ($j$){$j$};
\filldraw [color=teal](6,-1) circle (3pt) node[align=center,below,color=teal] ($j+1$){$j+1$};
\draw [color=teal](0,-1)--(2,-1) (5,-1)--(6,-1);
\draw [color=teal, dashed] (2,-1)--(5,-1);

\filldraw [color=teal](0,-2) circle (3pt) node[align=center,below,color=teal] (1){1};
\filldraw [color=teal](1,-2) circle (3pt) node[align=center,below,color=teal] (2){2};
\filldraw [color=teal](2,-2) circle (3pt) node[align=center,below,color=teal] (3){3};
\filldraw [color=teal](3,-2) circle (3pt) node[align=center,below,color=teal] (){};
\filldraw [color=teal](4,-2) circle (3pt) node[align=center,below,color=teal] (){};
\filldraw [color=teal](5,-2) circle (3pt) node[align=center,below,color=teal] ($j$){$j$};
\filldraw [color=teal](6,-2) circle (3pt) node[align=center,below,color=teal] ($j+1$){$j+1$};
\draw [color=teal](0,-2)--(2,-2) (5,-2)--(6,-2);
\draw [color=teal, dashed] (2,-2)--(5,-2);

\filldraw [color=cyan](0,0) circle (3pt) node[align=center,below,color=cyan] (1){1};
\filldraw [color=magenta](1,0) circle (3pt) node[align=center,below,color=magenta] (2){2};
\filldraw [color=magenta](2,0) circle (3pt) node[align=center,below,color=magenta] (3){3};
\filldraw [color=magenta](3,0) circle (3pt) node[align=center,below,color=magenta] (){};
\filldraw [color=magenta](4,0) circle (3pt) node[align=center,below,color=magenta] (){};
\filldraw [color=magenta](5,0) circle (3pt) node[align=center,below,color=magenta] ($j$){$j$};
\filldraw [color=magenta](6,0) circle (3pt) node[align=center,below,color=magenta] ($j+1$){$j+1$};
\draw [color=magenta] (1,0)--(2,0) (5,0)--(6,0);
\draw [color=magenta, dashed] (2,0)--(5,0);

\end{tikzpicture}
$$
\caption{\label{fig:ftp2example} The forest triples used to calculate $F_{P_2}(q)$.}
\end{figure}
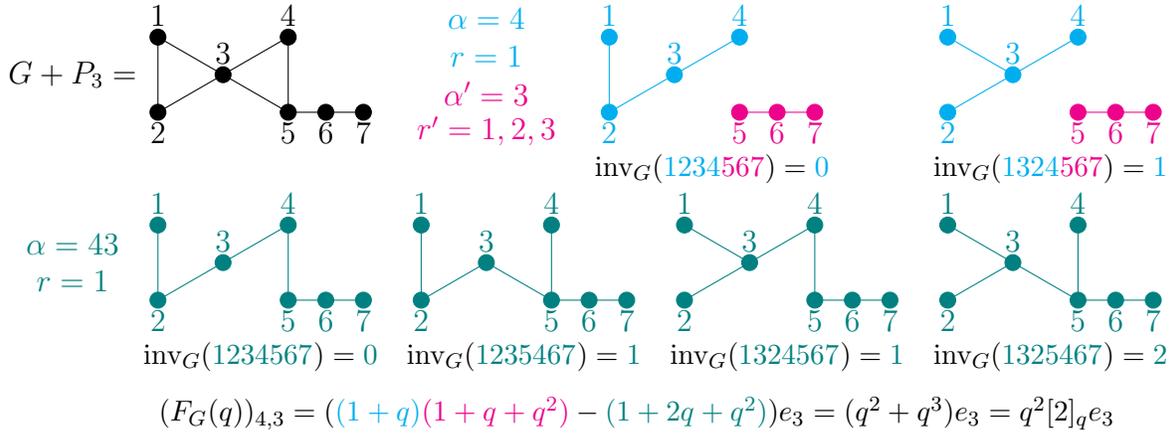
\begin{figure}
\begin{tikzpicture}
\draw (2.5,0) node (1) {$G+P_3=$};
\filldraw (3.634,0.5) circle (3pt) node[align=center,above] (1){1};
\filldraw (3.634,-0.5) circle (3pt) node[align=center,below] (2){2};
\filldraw (4.5,0) circle (3pt) node[align=center,above] (3){3};
\filldraw (5.366,0.5) circle (3pt) node[align=center,above] (4){4};
\filldraw (5.366,-0.5) circle (3pt) node[align=center,below] (5){5};
\filldraw (5.866,-0.5) circle (3pt) node[align=center,below] (6){6};
\filldraw (6.366,-0.5) circle (3pt) node[align=center,below] (7){7};
\draw (3.634,0.5)--(3.634,-0.5)--(4.5,0)--(3.634,0.5) (5.366,0.5)--(5.366,-0.5)--(4.5,0)--(5.366,0.5) (5.366,-0.5)--(6.366,-0.5);

\draw (2.5,-2.25) node (1){$\textcolor{teal}{\alpha=43}$};
\draw (2.5,-2.75) node (1) {$\textcolor{teal}{r=1}$};

\filldraw [color=teal](3.634,-2) circle (3pt) node[align=center,above,color=teal] (1){1};
\filldraw [color=teal](3.634,-3) circle (3pt) node[align=center,below,color=teal] (2){2};
\filldraw [color=teal](4.5,-2.5) circle (3pt) node[align=center,above,color=teal] (3){3};
\filldraw [color=teal](5.366,-2) circle (3pt) node[align=center,above,color=teal] (4){4};
\filldraw [color=teal](5.366,-3) circle (3pt) node[align=center,below,color=teal] (5){5};
\filldraw [color=teal](5.866,-3) circle (3pt) node[align=center,below,color=teal] (6){6};
\filldraw [color=teal](6.366,-3) circle (3pt) node[align=center,below,color=teal] (7){7};
\draw [color=teal] (3.634,-2)--(3.634,-3)--(4.5,-2.5) (4.5,-2.5)--(5.366,-2)--(5.366,-3)--(6.366,-3);
\draw (5,-3.75) node (){\small{$\text{inv}_G(\textcolor{teal}{1234567})=\textcolor{teal}0$}};

\filldraw [color=teal](7.134,-2) circle (3pt) node[align=center,above,color=teal] (1){1};
\filldraw [color=teal](7.134,-3) circle (3pt) node[align=center,below,color=teal] (2){2};
\filldraw [color=teal](8,-2.5) circle (3pt) node[align=center,above,color=teal] (3){3};
\filldraw [color=teal](8.866,-2) circle (3pt) node[align=center,above,color=teal] (4){4};
\filldraw [color=teal](8.866,-3) circle (3pt) node[align=center,below,color=teal] (5){5};
\filldraw [color=teal](9.366,-3) circle (3pt) node[align=center,below,color=teal] (6){6};
\filldraw [color=teal](9.866,-3) circle (3pt) node[align=center,below,color=teal] (7){7};
\draw [color=teal] (7.134,-2)--(7.134,-3)--(8,-2.5) (8.866,-2)--(8.866,-3) (8,-2.5)--(8.866,-3)--(9.866,-3);
\draw (8.5,-3.75) node (){\small{$\text{inv}_G(\textcolor{teal}{1235467})=\textcolor{teal}1$}};

\filldraw [color=teal](10.634,-2) circle (3pt) node[align=center,above,color=teal] (1){1};
\filldraw [color=teal](10.634,-3) circle (3pt) node[align=center,below,color=teal] (2){2};
\filldraw [color=teal](11.5,-2.5) circle (3pt) node[align=center,above,color=teal] (3){3};
\filldraw [color=teal](12.366,-2) circle (3pt) node[align=center,above,color=teal] (4){4};
\filldraw [color=teal](12.366,-3) circle (3pt) node[align=center,below,color=teal] (5){5};
\filldraw [color=teal](12.866,-3) circle (3pt) node[align=center,below,color=teal] (6){6};
\filldraw [color=teal](13.366,-3) circle (3pt) node[align=center,below,color=teal] (7){7};
\draw [color=teal] (10.634,-2)--(11.5,-2.5)--(10.634,-3) (11.5,-2.5)--(12.366,-2)--(12.366,-3)--(13.366,-3);
\draw (12,-3.75) node (){\small{$\text{inv}_G(\textcolor{teal}{1324567})=\textcolor{teal}1$}};

\filldraw [color=teal](14.134,-2) circle (3pt) node[align=center,above,color=teal] (1){1};
\filldraw [color=teal](14.134,-3) circle (3pt) node[align=center,below,color=teal] (2){2};
\filldraw [color=teal](15,-2.5) circle (3pt) node[align=center,above,color=teal] (3){3};
\filldraw [color=teal](15.866,-2) circle (3pt) node[align=center,above,color=teal] (4){4};
\filldraw [color=teal](15.866,-3) circle (3pt) node[align=center,below,color=teal] (5){5};
\filldraw [color=teal](16.366,-3) circle (3pt) node[align=center,below,color=teal] (6){6};
\filldraw [color=teal](16.866,-3) circle (3pt) node[align=center,below,color=teal] (7){7};
\draw [color=teal] (14.134,-2)--(15,-2.5)--(14.134,-3) (15.866,-2)--(15.866,-3) (15,-2.5)--(15.866,-3)--(16.866,-3);
\draw (15.5,-3.75) node (){\small{$\text{inv}_G(\textcolor{teal}{1325467})=\textcolor{teal}2$}};

\node at (8,0.75) {$\textcolor{cyan}{\alpha=4}$};
\node at (8,0.25) {$\textcolor{cyan}{r=1}$};
\node at (8,-0.25) {$\textcolor{magenta}{\alpha'=3}$};
\node at (8,-0.75) {$\textcolor{magenta}{r'=1,2,3}$};

\filldraw [color=cyan](9.634,0.5) circle (3pt) node[align=center,above,color=cyan] (1){1};
\filldraw [color=cyan](9.634,-0.5) circle (3pt) node[align=center,below,color=cyan] (2){2};
\filldraw [color=cyan](10.5,0) circle (3pt) node[align=center,above,color=cyan] (3){3};
\filldraw [color=cyan](11.366,0.5) circle (3pt) node[align=center,above,color=cyan] (4){4};
\filldraw [color=magenta](11.366,-0.5) circle (3pt) node[align=center,below,color=magenta] (5){5};
\filldraw [color=magenta](11.866,-0.5) circle (3pt) node[align=center,below,color=magenta] (6){6};
\filldraw [color=magenta](12.366,-0.5) circle (3pt) node[align=center,below,color=magenta] (7){7};
\draw [color=cyan](9.634,0.5)--(9.634,-0.5)--(10.5,0)--(11.366,0.5); \draw [color=magenta](11.366,-0.5)--(12.366,-0.5);
\draw (11,-1.25) node (){\small{$\text{inv}_G(\textcolor{cyan}{1234}\textcolor{magenta}{567})=\textcolor{cyan}0$}};

\filldraw [color=cyan](14.134,0.5) circle (3pt) node[align=center,above,color=cyan] (1){1};
\filldraw [color=cyan](14.134,-0.5) circle (3pt) node[align=center,below,color=cyan] (2){2};
\filldraw [color=cyan](15,0) circle (3pt) node[align=center,above,color=cyan] (3){3};
\filldraw [color=cyan](15.866,0.5) circle (3pt) node[align=center,above,color=cyan] (4){4};
\filldraw [color=magenta](15.866,-0.5) circle (3pt) node[align=center,below,color=magenta] (5){5};
\filldraw [color=magenta](16.366,-0.5) circle (3pt) node[align=center,below,color=magenta] (6){6};
\filldraw [color=magenta](16.866,-0.5) circle (3pt) node[align=center,below,color=magenta] (7){7};
\draw [color=cyan](14.134,-0.5)--(15,0)--(14.134,0.5) (15,0)--(15.866,0.5); 
\draw [color=magenta](15.866,-0.5)--(16.866,-0.5);
\draw (15.5,-1.25) node (){\small{$\text{inv}_G(\textcolor{cyan}{1324}\textcolor{magenta}{567})=\textcolor{cyan}1$}};
\draw (10,-4.5) node [](){\small{$(F_G(q))_{4,3}=(\textcolor{cyan}{(1+q)}\textcolor{magenta}{(1+q+q^2)}-\textcolor{teal}{(1+2q+q^2)})e_3=(q^2+q^3)e_3=q^2[2]_qe_3$}};
\end{tikzpicture}

\caption{\label{fig:ftmatrixexample} The forest triples used to calculate $(F_G(q))_{4,3}$.}
\end{figure}
\begin{example}\label{ex:bowtiematrix}
Let us calculate the entry $(F_G(q))_{4,3}$ for the bowtie graph $G$. The forest triples of $\text{FT}^{(4)}(G+P_3)$ are given in Figure \ref{fig:ftmatrixexample}. By calculating the signed weights, we have $(F_G(q))_{4,3}=q^2[2]_qe_3$. 
As an illustration, when $q=1$, we have 
\begin{equation*}
F_G=\left[\begin{matrix}
4e_4&2e_{41}+6e_5&e_{33}+3e_{51}+8e_6&2e_{43}+4e_{61}+10e_7&3e_{53}+5e_{71}+12e_8&\cdots\\
2e_3&e_{31}+4e_4&2e_{41}+6e_5&3e_{51}+8e_6&4e_{61}+10e_7&\cdots\\0&2e_3&e_{31}+4e_4&2e_{41}+6e_5&3e_{51}+8e_6&\cdots\\
2e_1&e_{11}&2e_3&4e_4&6e_5&\cdots\\
4&4e_1&e_{11}&0&0&\cdots\\
0&4&4e_1&e_{11}&0&\cdots\\
0&0&4&4e_1&e_{11}&\cdots\\
0&0&0&4&4e_1&\cdots\\
0&0&0&0&4&\cdots\\
\vdots&\vdots&\vdots&\vdots&\vdots&\ddots
\end{matrix}\right].
\end{equation*}
\end{example}

Although $F_G$ is an infinite matrix, we now show that every column contains finitely many nonzero entries.

\begin{proposition}\label{prop:ftmatrixzeroes}
Let $G=([n],E)$ be a graph. If $i\geq n+j$, or if $j\geq i\geq n\geq 2$, then $\text{FT}^{(i)}(G+P_j)$ is empty, so $(F_G)_{i,j}=0$.
\end{proposition}

\begin{proof}
If $i\geq n+j$, then we cannot have $\alpha_1=i$ because $G+P_j$ has only $(n+j-1)$ vertices. If $j\geq i\geq n\geq 2$, then since $|T|=|\alpha|\geq \alpha_1 =i\geq n$, we must have $n\in T$ and $\mathcal T=\mathcal T'$. Now the tree $T$ must contain the $(j+1)$ vertices $1,n,n+1,\ldots,n+j-1$ and we must have $\alpha_1=i$ and $\alpha_\ell\geq j$. But now if $\ell(\alpha)\geq 2$, we have $|\alpha|\geq i+j\geq n+j$, while $G+P_j$ has only $(n+j-1)$ vertices, and if $\ell(\alpha)=1$, then $|\alpha|=i\leq j$, but $T$ has at least $(j+1)$ vertices.
\end{proof}

We will often be interested in the first column of $F_G$, so we introduce the following notation.

\begin{definition}
We define the infinite row vectors \begin{equation*}\vec v=\left[\begin{matrix}e_1&2e_2&3e_3&\cdots\end{matrix}\right], \ \vec v(q)=\left[\begin{matrix}e_1&[2]_qe_2&[3]_qe_3&\cdots\end{matrix}\right],\text{ and }\vec w=\left[\begin{matrix}1&0&0&\cdots\end{matrix}\right].\end{equation*} 
For a graph $G=([n],E)$, we define the symmetric function \begin{equation*}X^{(i)}_G(\bm x)=(F_G)_{i,1}=\sum_{\substack{\mathcal F\in\text{FT}(G)\\\alpha^{(1)}_1=i, \ r_1=1}}\text{sign}(\mathcal F)e_{\text{type}'(\mathcal F)}.\end{equation*} 
If $G$ is an NUIG, we define $X^{(i)}_G(\bm x;q)=(F_G(q))_{i,1}$.
\end{definition}

Now we can calculate $X_G(\bm x)$ as a matrix multiplication. Because the columns of $F_G$ have finitely many nonzero entries by Proposition~\ref{prop:ftmatrixzeroes}, the matrix multiplications make sense.

\begin{proposition}\label{prop:xmatrix}
We have
\begin{equation}\label{eq:xmatrix}
X_G(\bm x)=\vec v F_G\vec w^T.
\end{equation}
Similarly, if $G$ is an NUIG, then $X_G(\bm x;q)=\vec v(q)F_G(q)\vec w^T$.
\end{proposition}

\begin{proof}
We have
\begin{equation*}
X_G(\bm x)=\sum_{i\geq 1}ie_i\left(\sum_{\mathcal F\in\text{FT}^{(i)}(G+P_1)}\text{sign}(\mathcal F)e_{\text{type}'(\mathcal F)}\right)=\sum_{i\geq 1}ie_i(F_G)_{i,1}=\vec vF_G\vec w^T,
\end{equation*}
where the factor of $i$ accounts for the possible choices of $1\leq r_i\leq\alpha^{(1)}_1=i$. The proof of the second statement is similar.
\end{proof}

In other words, $X_G(\bm x)$ is the dot product between $\vec v$ and the first column of $F_G$, consisting of the $X^{(i)}_G(\bm x)$. In particular, if $F_G$ (respectively $F_G(q)$ if $G$ is an NUIG) has all $e$-positive entries, then $X_G(\bm x)$ (respectively $X_G(\bm x;q)$) is $e$-positive.

\begin{example}
For the bowtie graph $G$, we can calculate $X_G(\bm x;q)$ as
\begin{align*}
X_G(\bm x;q)&=\vec v(q)F_G(q)\vec w^T\\\nonumber&=e_1\cdot q^2[2]_q[2]_qe_4+[2]_qe_2\cdot q^2[2]_qe_3+[4]_qe_4\cdot q[2]_qe_1+[5]_qe_5\cdot[2]_q[2]_q\\\nonumber&=q^2[2]_q[2]_qe_{32}+q[3]_q[2]_qe_{41}+[5]_q[2]_q[2]_qe_5.
\end{align*}\end{example}

We now prove that the operation of gluing graphs corresponds to multiplying matrices.

\begin{proposition}\label{prop:ftmult}
For graphs $G$ and $H$, we have
\begin{equation*}
F_{G+H}=F_GF_H.
\end{equation*}
Moreover, if $G$ and $H$ are NUIGs, then $F_{G+H}(q)=F_G(q)F_H(q)$.
\end{proposition}

\begin{proof}
We will find a bijection
\begin{equation*}
\text{break}:\text{FT}^{(i)}(G+H+P_j)\to\bigsqcup_{k\geq 1}\text{FT}^{(i)}(G+P_k)\times\text{FT}^{(k)}(H+P_j)
\end{equation*}
such that if $\text{break}(\mathcal F)=(\mathcal F\vert_G,\mathcal F\vert_H)$, then 
\begin{equation}\label{eq:breakprops}
\text{sign}(\mathcal F)=\text{sign}(\mathcal F\vert_G)\text{sign}(\mathcal F\vert_H)\text{ and }e_{\text{type}'(\mathcal F)}=e_{\text{type}'(\mathcal F\vert_G)}e_{\text{type}'(\mathcal F\vert_H)}.
\end{equation}
Then we will have, as desired, that
\begin{align*}
(F_{G+H})_{i,j}&=\sum_{\mathcal F\in\text{FT}^{(i)}(G+H+P_j)}\text{sign}(\mathcal F)e_{\text{type}'(\mathcal F)}\\&=\sum_{k\geq 1}\sum_{\mathcal F\vert_G\in\text{FT}^{(i)}(G+P_k)}\text{sign}(\mathcal F\vert_G)e_{\text{type}'(\mathcal F\vert_G)}\sum_{\mathcal F\vert_H\in\text{FT}^{(k)}(H+P_j)}\text{sign}(\mathcal F\vert_H)e_{\text{type}'(\mathcal F\vert H)}\\&=\sum_{k\geq 1}(F_G)_{i,k}(F_H)_{k,j}=(F_GF_H)_{i,j}.
\end{align*}
We will also show that if $G$ and $H$ are NUIGs, then $\text{weight}(\mathcal F)=\text{weight}(\mathcal F\vert_G)+\text{weight}(\mathcal F\vert_H)$, from which the second statement will follow.\\

Let $\mathcal F\in\text{FT}^{(i)}(G+H+P_j)$. Let $n=|G|$, $n'=|H|$, and let $\mathcal T=(T,\alpha,r)$ be the tree triple of $\mathcal F$ with $n\in T$. Let $T\vert_G$ and $T\vert_H$ be the trees obtained by restricting $T$ to the vertices at most $n$ and at least $n$ respectively. Because $|\alpha|=|T|\geq |T\vert_G|$, there is some minimal $t$ with $\alpha_1+\cdots+\alpha_t\geq |T\vert_G|$; let 
\begin{equation*}
k=\alpha_1+\cdots+\alpha_t-|T\vert_G|+1.
\end{equation*}
We now define the tree triples
\begin{equation*}
\mathcal T\vert_G=(T\vert_G\cup P_k,\alpha_1\cdots\alpha_t,r)\text{ and }\mathcal T\vert_H=(T\vert_H,k \ \alpha_{t+1}\cdots\alpha_\ell,1),
\end{equation*}
where $T\vert_G\cup P_k$ is obtained by attaching the path on vertices $n,\ldots,n+k-1$ to $T\vert_G$. We now define $\mathcal F\vert_G$ to be the forest triple consisting of $\mathcal T\vert_G$ and the tree triples of $\mathcal F$ whose trees contain a vertex less than $n$. We define $\mathcal F\vert_H$ to be the forest triple consisting of $\mathcal T\vert_H$ and the tree triples of $\mathcal F$ whose trees contain a vertex greater than $n$, but we subtract $(n-1)$ from all of the vertices so that they lie in $[n'+k-1]$. Some examples are given in Figure \ref{fig:ftbreakexample}.\\
\begin{figure}
\begin{align*}
&
\begin{tikzpicture}
\node at (5.5,1.5) (){$P_3$};
\node at (2,1.5) (){$G_{}$};
\node at (4,1.5) (){$H_{}$};
\node at (10.5,1.5) (){$G_{}$};
\node at (12,1.5) (){$P_3$};
\node at (14.5,1.5) (){$H_{}$};
\node at (16,1.5) (){$P_3$};
\node [color=teal] at (2,-1.5) () {$\mathcal T$};
\node [color=magenta] at (10.5,-1.5) () {$\mathcal T\vert_G$};
\node [color=cyan] at (14.5,-1.5) () {$\mathcal T\vert_H$};
\node [color=teal] at (2,-2) () {$\alpha=241$};
\node [color=magenta] at (10.5,-2) () {$\alpha=24$};
\node [color=cyan] at (14.5,-2) () {$\alpha=31$};
\draw [dashed,thin,opacity=0.4] (1,0)--(1.5,-0.866) (1.5,0.866)--(2.5,0.866) (1,0)--(2.5,0.866)--(2.5,-0.866)--(1,0) (1.5,0.866)--(1.5,-0.866)--(3,0)--(1.5,0.866) (1.5,0.866)--(2.5,-0.866) (1.5,-0.866)--(2.5,0.866) (3,0)--(4,0)--(4.5,0.866)--(5,0)--(4,0) (9.5,0)--(10,-0.866) (10,0.866)--(11,0.866) (9.5,0)--(11,0.866)--(11,-0.866)--(9.5,0) (10,0.866)--(10,-0.866)--(11.5,0)--(10,0.866) (10,0.866)--(11,-0.866) (10,-0.866)--(11,0.866) (13.5,0)--(14.5,0)--(15,0.866)--(15.5,0)--(14.5,0);
\draw [color=teal] (2.5,0.866)--(3,0)--(2.5,-0.866)--(1.5,-0.866) (3,0)--(3.5,0.866)--(4.5,0.866) (3.5,0.866)--(4,0);
\filldraw [color=gray] (1,0) circle (3pt) node[align=center,below] (1){1};
\filldraw [color=gray] (1.5,0.866) circle (3pt) node[align=center,above] (2){2};
\filldraw [color=teal] (1.5,-0.866) circle (3pt) node[align=center,below] (3){3};
\filldraw [color=teal] (2.5,0.866) circle (3pt) node[align=center,above] (4){4};
\filldraw [color=teal] (2.5,-0.866) circle (3pt) node[align=center,below] (5){5};
\filldraw (3,0) circle (3pt) node[align=center,below] (6){6};
\filldraw [color=teal] (3.5,0.866) circle (3pt) node[align=center,above] (7){7};
\filldraw [color=teal] (4,0) circle (3pt) node[align=center,below] (8){8};
\filldraw [color=teal] (4.5,0.866) circle (3pt) node[align=center,above] (9){9};
\filldraw [color=gray] (5,0) circle (3pt) node[align=center,below] (10){10};
\filldraw [color=gray] (5.5,0) circle (3pt) node[align=center,below] (11){11};
\filldraw [color=gray] (6,0) circle (3pt) node[align=center,below] (12){12};
\draw [color=gray] (5,0)--(6,0) (1,0)--(1.5,0.866);
\draw [arrows = {-Stealth[scale=2]}] (6.5,0)--(9,0);
\draw [arrows = {-Stealth[scale=2]}] (9,0)--(6.5,0);
\draw [color=gray] (9.5,0)--(10,0.866) (16.5,0)--(15.5,0);
\draw [color=magenta] (10,-0.866)--(11,-0.866)--(11.5,0)--(11,0.866) (11.5,0)--(12.5,0);
\draw [color=cyan] (13.5,0)--(14,0.866)--(15,0.866) (14.5,0)--(14,0.866);
\filldraw [color=gray] (9.5,0) circle (3pt) node[align=center,below] (1){1};
\filldraw [color=gray] (10,0.866) circle (3pt) node[align=center,above] (2){2};
\filldraw [color=magenta] (10,-0.866) circle (3pt) node[align=center,below] (3){3};
\filldraw [color=magenta] (11,0.866) circle (3pt) node[align=center,above] (4){4};
\filldraw [color=magenta] (11,-0.866) circle (3pt) node[align=center,below] (5){5};
\filldraw (11.5,0) circle (3pt) node[align=center,below] (6){6};
\filldraw [color=magenta] (12,0) circle (3pt) node[align=center,below] (7){7};
\filldraw [color=magenta] (12.5,0) circle (3pt) node[align=center,below] (8){8};
\node at (13,-0.3) {\huge{,}};
\filldraw (13.5,0) circle (3pt) node[align=center,below] (1){1};
\filldraw [color=cyan] (14,0.866) circle (3pt) node[align=center,above] (2){2};
\filldraw [color=cyan] (14.5,0) circle (3pt) node[align=center,below] (3){3};
\filldraw [color=cyan] (15,0.866) circle (3pt) node[align=center,above] (4){4};
\filldraw [color=gray] (15.5,0) circle (3pt) node[align=center,below] (5){5};
\filldraw [color=gray] (16,0) circle (3pt) node[align=center,below] (6){6};
\filldraw [color=gray] (16.5,0) circle (3pt) node[align=center,below] (7){7};
\end{tikzpicture}\\&
\begin{tikzpicture}
\node at (5.5,1.5) (){$P_3$};
\node at (2,1.5) (){$G_{}$};
\node at (4,1.5) (){$H_{}$};
\node at (10.5,1.5) (){$G_{}$};
\node at (11.75,1.5) (){$P_2$};
\node at (14,1.5) (){$H_{}$};
\node at (15.5,1.5) (){$P_3$};
\node [color=teal] at (2,-1.5) () {$\mathcal T$};
\node [color=magenta] at (10.5,-1.5) () {$\mathcal T\vert_G$};
\node [color=cyan] at (14,-1.5) () {$\mathcal T\vert_H$};
\node [color=teal] at (2,-2) () {$\alpha=52$};
\node [color=magenta] at (10.5,-2) () {$\alpha=5$};
\node [color=cyan] at (14,-2) () {$\alpha=22$};
\draw [dashed,thin,opacity=0.4] (1,0)--(1.5,-0.866) (1.5,0.866)--(2.5,0.866) (1,0)--(2.5,0.866)--(2.5,-0.866)--(1,0) (1.5,0.866)--(1.5,-0.866)--(3,0)--(1.5,0.866) (1.5,0.866)--(2.5,-0.866) (1.5,-0.866)--(2.5,0.866) (3,0)--(4,0)--(4.5,0.866)--(5,0)--(4,0) (9.5,0)--(10,-0.866) (10,0.866)--(11,0.866) (9.5,0)--(11,0.866)--(11,-0.866)--(9.5,0) (10,0.866)--(10,-0.866)--(11.5,0)--(10,0.866) (10,0.866)--(11,-0.866) (10,-0.866)--(11,0.866) (13,0)--(14,0)--(14.5,0.866)--(15,0)--(14,0);
\draw [color=teal] (2.5,0.866)--(3,0)--(2.5,-0.866)--(1.5,-0.866) (3,0)--(3.5,0.866)--(4.5,0.866) (3.5,0.866)--(4,0);
\filldraw [color=gray] (1,0) circle (3pt) node[align=center,below] (1){1};
\filldraw [color=gray] (1.5,0.866) circle (3pt) node[align=center,above] (2){2};
\filldraw [color=teal] (1.5,-0.866) circle (3pt) node[align=center,below] (3){3};
\filldraw [color=teal] (2.5,0.866) circle (3pt) node[align=center,above] (4){4};
\filldraw [color=teal] (2.5,-0.866) circle (3pt) node[align=center,below] (5){5};
\filldraw (3,0) circle (3pt) node[align=center,below] (6){6};
\filldraw [color=teal] (3.5,0.866) circle (3pt) node[align=center,above] (7){7};
\filldraw [color=teal] (4,0) circle (3pt) node[align=center,below] (8){8};
\filldraw [color=teal] (4.5,0.866) circle (3pt) node[align=center,above] (9){9};
\filldraw [color=gray] (5,0) circle (3pt) node[align=center,below] (10){10};
\filldraw [color=gray] (5.5,0) circle (3pt) node[align=center,below] (11){11};
\filldraw [color=gray] (6,0) circle (3pt) node[align=center,below] (12){12};
\draw [color=gray] (5,0)--(6,0) (1,0)--(1.5,0.866);
\draw [arrows = {-Stealth[scale=2]}] (6.5,0)--(9,0);
\draw [arrows = {-Stealth[scale=2]}] (9,0)--(6.5,0);
\draw [color=gray] (9.5,0)--(10,0.866) (15,0)--(16,0);
\draw [color=magenta] (10,-0.866)--(11,-0.866)--(11.5,0)--(11,0.866) (11.5,0)--(12,0);
\draw [color=cyan] (13,0)--(13.5,0.866)--(14.5,0.866) (14,0)--(13.5,0.866);
\filldraw [color=gray] (9.5,0) circle (3pt) node[align=center,below] (1){1};
\filldraw [color=gray] (10,0.866) circle (3pt) node[align=center,above] (2){2};
\filldraw [color=magenta] (10,-0.866) circle (3pt) node[align=center,below] (3){3};
\filldraw [color=magenta] (11,0.866) circle (3pt) node[align=center,above] (4){4};
\filldraw [color=magenta] (11,-0.866) circle (3pt) node[align=center,below] (5){5};
\filldraw (11.5,0) circle (3pt) node[align=center,below] (6){6};
\filldraw [color=magenta] (12,0) circle (3pt) node[align=center,below] (7){7};
\node at (12.5,-0.3) {\huge{,}};
\filldraw (13,0) circle (3pt) node[align=center,below] (1){1};
\filldraw [color=cyan] (13.5,0.866) circle (3pt) node[align=center,above] (2){2};
\filldraw [color=cyan] (14,0) circle (3pt) node[align=center,below] (3){3};
\filldraw [color=cyan] (14.5,0.866) circle (3pt) node[align=center,above] (4){4};
\filldraw [color=gray] (15,0) circle (3pt) node[align=center,below] (5){5};
\filldraw [color=gray] (15.5,0) circle (3pt) node[align=center,below] (6){6};
\filldraw [color=gray] (16,0) circle (3pt) node[align=center,below] (7){7};
\end{tikzpicture}
\end{align*}
\caption{\label{fig:ftbreakexample} Some examples of the map $$\text{break}:\text{FT}^{(2)}(G+H+P_3)\to\bigsqcup_{k\geq 1}\text{FT}^{(2)}(G+P_k)\times\text{FT}^{(k)}(H+P_3).$$}
\end{figure}

We first note that the first part of $\alpha^{(1)}$ has not been affected. By minimality of $t$, we have $\alpha_1+\cdots+\alpha_{t-1}<|T\vert_G|$ and therefore $\alpha_t\geq k$, so we indeed have $\mathcal F\vert_G\in\text{FT}^{(i)}(G+P_k)$. By construction, we have $\mathcal F\vert_H\in\text{FT}^{(k)}(G+P_j)$ and \eqref{eq:breakprops} holds. Given $\mathcal F\vert_G\in\text{FT}^{(i)}(G+P_k)$ and $\mathcal F\vert_H\in\text{FT}^{(k)}(G+P_j)$, the inverse map is given as follows. We first add $(n-1)$ to all of the vertices of trees in $\mathcal F\vert_H$ so that they lie in $[n,n+n'+k-2]$. Then letting $\mathcal T=(T,\alpha,r)\in\mathcal F\vert_G$ be the tree triple with $n\in T$ and letting $\mathcal T'=(T',\alpha',1)\in\mathcal F\vert_H$ be the tree triple with $n\in T'$, we define the tree triple 
\begin{equation}
\mathcal T\cup\mathcal T'=((T\setminus P_k)\cup T',\alpha_1\cdots\alpha_\ell\cdot\alpha'_2\cdots\alpha'_\ell,r),
\end{equation} where $(T\setminus P_k)\cup T'$ is obtained by removing the path on vertices $n,\ldots,n+k-1$ from $T$, and then gluing to $T'$. Then $\mathcal F$ is the forest triple consisting of $\mathcal T\cup\mathcal T'$ and the other tree triples in $\mathcal F\vert_G$ and $\mathcal F\vert_H$. \\

Finally, if $G$ and $H$ are NUIGs, then \begin{align*}\text{inv}_{G+H+P_j}(\text{list}(T))&=\text{inv}_G(\text{list}(T\vert_G))+\text{inv}_{H+P_j}(\text{list}(T\vert_H))\\&=\text{inv}_{G+P_k}(\text{list}(T\vert_G\cup P_k))+\text{inv}_{H+P_j}(\text{list}(T\vert_H)\end{align*}
because $(G+H+P_j)$-inversions in $\text{list}(T)$ appear either in $\text{list}(T\vert_G)$ or $\text{list}(T\vert_H)$ and the path adds no new inversions. So we have $\text{weight}(\mathcal F)=\text{weight}(\mathcal F\vert_G)+\text{weight}(\mathcal F\vert_H)$.
\end{proof}

\begin{remark}
By Proposition \ref{prop:xmatrix}, we can now calculate the chromatic symmetric function $X_{G+H}(\bm x)$ as 
\begin{equation*}
X_{G+H}(\bm x)=\vec vF_{G+H}\vec w^T=\vec vF_GF_H\vec w^T,
\end{equation*}
that is, $X_{G+H}(\bm x)$ is the dot product of $\vec v F_G$, the row vector whose $j$-th entry is $X_{G+P_j}(\bm x)$, and $F_H\vec w^T$, the column vector whose $i$-th entry is $X^{(i)}_H(\bm x)$. So we can calculate $X_{G+H}(\bm x)$ if we know $X_{G+P_j}(\bm x)$ and $X^{(i)}_H(\bm x)$ for every $i,j$. 
\end{remark}

\begin{example}
Because the bowtie graph $G$ from Figure \ref{fig:ftmatrixexample} is the sum of two complete graphs $G=K_3+K_3$, Proposition \ref{prop:ftmult} tells us that $F_G(q)=F_{K_3}(q)^2$. 
\end{example}
\begin{remark}
Proposition \ref{prop:ftp1} and Proposition \ref{prop:ftmult} tell us that the forest triple matrix is a \emph{monoid homomorphism} on graphs under the gluing operation.
\end{remark}

\begin{proposition}\label{prop:ftpn}
For the path $P_n$, we have
\begin{equation}\label{eqn:ftm:fpn}
(F_{P_n}(q))_{i,j}=\sum_{\alpha\vDash n+j-i-1, \ \alpha_\ell\geq j}([\alpha_1]_q-1)\cdots([\alpha_\ell]_q-1)e_{\text{sort}(\alpha)}.
\end{equation}
Note that if $i=n+j-1$, we are including the empty composition in our sum. The summand is the empty product, so $(F_{P_n}(q))_{n+j-1,j}=1$.
\end{proposition}

\begin{proof}
If $n\leq 2$, then the result follows from Proposition~\ref{prop:ftp1} and Proposition~\ref{prop:ftp2}. For $n\geq 3$, because $P_n=P_{n-1}+P_2$, we can use induction and Proposition \ref{prop:ftmult}. Indeed, defining the polynomial $w_\alpha(q)=([\alpha_1]_q-1)\cdots([\alpha_\ell]_q-1)$, we have
\begin{align*}
(F_{P_n}(q))_{i,j}&=(F_{P_{n-1}}(q)F_{P_2}(q))_{i,j}=(F_{P_{n-1}}(q))_{i,1}q[j-1]_qe_j+(F_{P_{n-1}}(q))_{i,j+1}\\&=\sum_{\alpha\vDash n-1-i}w_\alpha(q)e_{\text{sort}(\alpha)}([j]_q-1)e_j+\sum_{\alpha\vDash n+j-i-1, \ \alpha_\ell\geq j+1}w_\alpha(q)\\&=\sum_{\substack{\alpha\vDash n+j-i-1\\\alpha_\ell=j}}w_\alpha(q)e_{\text{sort}(\alpha)}+\sum_{\substack{\alpha\vDash n+j-i-1\\\alpha_\ell\geq j+1}}w_\alpha(q)e_{\text{sort}(\alpha)}=\sum_{\substack{\alpha\vDash n+j-i-1\\\alpha_\ell\geq j}}w_\alpha(q)e_{\text{sort}(\alpha)}.
\end{align*}
\end{proof}
\begin{remark} We could also calculate $F_{P_n}(q)$ by using the sign-reversing involution in \cite[Proposition 3.7]{qforesttriples} and enumerating the fixed points.\end{remark}
\begin{example}
For the path $P_5$, we have
\begin{equation*}
F_{P_5}(q)=\left[\begin{matrix}
q^2e_{22}+q[3]_qe_4&2q^2[2]e_{32}+q[4]_qe_5&q^2[2]_q[2]_qe_{33}+q[3]_qe_{42}+q[5]_qe_6&\cdots\\
q[2]_qe_3&q^2e_{22}+q[3]_qe_4&q^2[2]_qe_{32}+q[4]_qe_5&\cdots\\
qe_2&q[2]_qe_3&q[3]_qe_4&\cdots\\
0&qe_2&q[2]_qe_3&\cdots\\
1&0&0&\cdots\\
0&1&0&\cdots\\
0&0&1&\cdots\\
\vdots&\vdots&\vdots&\ddots
\end{matrix}\right].
\end{equation*}
\end{example}

In previous work, the authors found sign-reversing involutions on forest triples. This allows us to calculate some more forest triple matrices.

\begin{proposition}\label{prop:ftkn} \cite[Theorem 4.10]{qforesttriples}
For the complete graph $K_n$, we have
\begin{equation*}
(F_{K_n}(q))_{i,j}=\begin{cases}
q^{n-1}[n-2]_q![j-i]_qe_{n-i+j-1}&\text{ if }1\leq i\leq\min\{n-1,j-1\},\\
q^{n-i+j-1}[n-2]_q![i-j]_qe_{n-i+j-1}&\text{ if }\max\{n,j+1\}\leq i\leq n+j-1,\\
0&\text{ otherwise.}
\end{cases}
\end{equation*}
For the almost-complete graph $K_n'$ obtained from $K_n$ by removing the edge $(1,n)$, we have
\begin{equation*}
(F_{K_n'}(q))_{i,j}=\begin{cases}
q^{n-2}[n-2]_q![j-(i-1)]_qe_{n-i+j-1}&\text{ if }1\leq i\leq\min\{n-1,j\},\\
q^{n-i+j-1}[n-2]_q![(i-1)-j]_qe_{n-i+j-1}&\text{ if }\max\{n,j+1\}\leq i\leq n+j-1,\\
0&\text{ otherwise.}
\end{cases}
\end{equation*}
Note that the entries of $F_{K_n}(q)$ and $F_{K_n'}(q)$ are all $e$-positive.
\end{proposition}
\begin{example} For the complete graph $K_3$, we have
\begin{equation*}
F_{K_3}(q)=\left[\begin{matrix}
0&q^2e_3&q^2[2]_qe_4&q^2[3]_qe_5&q^2[4]_qe_6&\cdots\\
0&0&q^2e_3&q^2[2]_qe_4&q^2[3]_qe_5&\cdots\\
[2]_q&qe_1&0&0&0&\cdots\\
0&[2]_q&qe_1&0&0&\cdots\\
0&0&[2]_q&qe_1&0&\cdots\\
0&0&0&[2]_q&qe_1&\cdots\\
\vdots&\vdots&\vdots&\vdots&\vdots&\ddots
\end{matrix}\right].
\end{equation*}
\end{example}

\begin{proposition}\label{prop:ftcn}
\cite[Lemma 6.1]{cyclechains}
For the cycle graph $C_n$, we have
\begin{equation*}
(F_{C_n})_{i,j}=\sum_{\substack{\alpha\vDash n-i+j-1\\\alpha_1\geq j\geq i+1\text{ or }0\leq\alpha_1\leq j-1\leq i-1}}|j-i|(\alpha_2-1)\cdots(\alpha_\ell-1)e_{\text{sort}(\alpha)},
\end{equation*}
where we sum over nonempty weak compositions where only $\alpha_1$ is allowed to be $0$. Note that the entries of $F_{C_n}$ are all $e$-positive.
\end{proposition}

By Proposition \ref{prop:xmatrix} and Proposition \ref{prop:ftmult}, we immediately recover some known $e$-positivity results for graphs obtained by gluing a sequence of these graphs at single vertices.

\begin{corollary}\hspace{0pt}
\begin{enumerate}
\item \cite[Corollary 4.12]{qforesttriples} If $G$ is a sum of complete graphs and almost-complete graphs, then $X_G(\bm x;q)$ is $e$-positive. 
\item \cite[Theorem 6.5]{cyclechains} If $G$ is a sum of complete graphs, almost-complete graphs, and cycles, then $X_G(\bm x)$ is $e$-positive.
\end{enumerate}
\end{corollary}
Note that in Part (2), such graphs $G$ are not NUIGs, so $X_G(\bm x;q)$ may fail to be symmetric.

\begin{proof}
This follows from Propositions \ref{prop:xmatrix}, \ref{prop:ftmult}, \ref{prop:ftkn}, and \ref{prop:ftcn}.
\end{proof}

We now give an alternative way of calculating the forest triple matrix. This will help us compare the forest triple matrix to the tableau matrix that we will define in the next section.

\begin{proposition}\label{prop:ftmatrixkj}
The forest triple matrix of $G$ satisfies
\begin{equation*}
(F_G)_{i,j}=\frac 1{(j-1)!}(F_{G+K_j})_{i,1}=\frac 1{(j-1)!}X^{(i)}_{G+K_j}(\bm x)=\frac 1{(j-1)!}\sum_{\substack{\mathcal F\in\text{FT}(G+K_j)\\\alpha^{(1)}_1=i, \ r_1=1}}\text{sign}(\mathcal F)e_{\text{type}'(\mathcal F)}.
\end{equation*}
If $G$ is an NUIG, then we have $(F_G(q))_{i,j}=\frac 1{[j-1]_q!}(F_{G+K_j}(q))_{i,j}=\frac 1{[j-1]_q!}X^{(i)}_{G+K_j}(\bm x;q)$.
\end{proposition}

\begin{proof}
We need only prove the first equality because the others follow by definition. In Proposition \ref{prop:ftkn}, we saw that the first column of $F_{K_j}$ has a $(j-1)!$ in the $j$-th position and is $0$ elsewhere. 
Therefore, by Proposition \ref{prop:ftmult}, we have
\begin{equation*}
(F_G)_{i,j}=\frac 1{(j-1)!}(F_GF_{K_j})_{i,1}=\frac 1{(j-1)!}(F_{G+K_j})_{i,1}.
\end{equation*}
The proof of the second statement is similar.
\end{proof}

\begin{figure}
$$\begin{tikzpicture}
\node at (-2,0) {$G=$};
\draw (-1,0)--(1,0) (-0.5,-0.866)--(0.5,0.866) (0.5,-0.866)--(-0.5,0.866) (-1,0)--(-0.5,0.866) (-0.5,-0.866)--(0.5,-0.866) (0.5,0.866)--(1,0);
\filldraw [color=red] (0,0) circle (3pt) node[align=center,above] ($x$){$x$};
\filldraw (0.5,0.866) circle (3pt) node[align=center,below] (){};
\filldraw (0.5,-0.866) circle (3pt) node[align=center,below] (){};
\filldraw (-0.5,0.866) circle (3pt) node[align=center,below] (){};
\filldraw (-0.5,-0.866) circle (3pt) node[align=center,below] (){};
\filldraw (1,0) circle (3pt) node[align=center,below] (){};
\filldraw (-1,0) circle (3pt) node[align=center,below] (){};
\end{tikzpicture}$$
\caption{\label{fig:windmill} The graph obtained by gluing three triangles at a common vertex $x$.}
\end{figure}
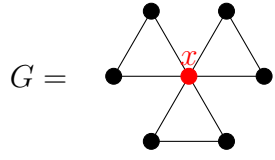

\begin{remark}
We could also consider the operation of gluing several graphs at a fixed vertex $x$ of $G$. We can still express the chromatic symmetric function as a matrix multiplication, although it is usually not $e$-positive. We can define the subset of forest triples
\begin{equation*}
\text{FT}^{(i)}((G,x)+P_j)=\{\mathcal F\in\text{FT}((G,x)+P_j): \ \alpha_1=i, r=1, P_j\subseteq T, \alpha_\ell\geq j\},
\end{equation*}
where $(G,x)+P_j$ is obtained by attaching the path of length $j$ at vertex $x$ of $G$ and $\mathcal T=(T,\alpha,r)$ is the tree triple of $\mathcal F$ with $x\in T$. Now if we define the infinite matrix
\begin{equation*}
(F_{(G,x)})_{i,j}=\sum_{\mathcal F\in\text{FT}^{(i)}((G,x)+P_j)}\text{sign}(\mathcal F)e_{\text{type}'(\mathcal F)},
\end{equation*}
and we define $(G,x)+(H,y)$ by gluing vertices $x$ and $y$ together, then the same arguments will give us that
\begin{equation*}
X_G(\bm x)=\vec vF_{(G,x)}\vec w^T\text{ and }F_{(G,x)+(H,y)}=F_{(G,x)}F_{(H,y)}.
\end{equation*}
For example, the graph $G$ in Figure \ref{fig:windmill} is obtained by gluing three triangles at a common vertex $x$, so we can calculate its chromatic symmetric function as $X_G(\bm x)=\vec vF_{(K_3,1)}^3\vec w^T$.
\end{remark}

\section{Tableau matrix}\label{section:tabmatrix}

In this section, we use Hikita's probabilities on standard Young tableaux to define a matrix. We will show that this matrix in fact coincides with the forest triple matrix. This connection will allow us to better understand both Hikita's tableaux and forest triples.

\begin{definition}
Let $G=([n],E)$ be an NUIG and consider the subset of $\text{SYT}(G+K_j)$
\begin{equation*}
\text{SYT}^{(i)}(G+K_j)=\{T\in\text{SYT}(G+K_j): \text{ the entry }n+j-1\text{ is in column }i\}.
\end{equation*} We now define the infinite matrix $T_G(q)$ by
\begin{equation*}
(T_G(q))_{i,j}=\frac 1{[i]_q[j-1]_q!}\sum_{T\in\text{SYT}^{(i)}(G+K_j)}c_T(G+K_j)e_{\text{shape}'(T)}.
\end{equation*}
Note that if we evaluate at $q=1$, the entries of $T_G(1)$ are all $e$-positive.
\end{definition}

We first calculate some examples.

\begin{proposition}\label{prop:tabp1}
For the graph $P_1$ with one vertex, $T_{P_1}(q)$ is the infinite identity matrix $I$.
\end{proposition}

\begin{proof}
We are considering standard Young tableaux $T\in\text{SYT}(P_1+K_j)$. By Proposition~\ref{prop:tabfirstrow}, there is only one such tableau $T$, with all $j$ entries on the first row and $c_T(K_j)=[j]_q!$.
\end{proof}

\begin{proposition}\label{prop:tabp2}
For the two-vertex path $P_2$, we have
\begin{equation}\label{eq:tabp2}
(T_{P_2}(q))_{i,j}=\begin{cases}
q[j-1]_qe_j&\text{ if }i=1,\\
1&\text{ if }i=j+1,\\
0&\text{ otherwise.}
\end{cases}
\end{equation}
Note that $T_{P_2}(q)=F_{P_2}(q)$. 
\end{proposition}
\begin{proof}
We are considering standard Young tableaux $T\in\text{SYT}(P_2+K_j)$. By Proposition~\ref{prop:tabfirstrow}, $T$ must have the entries $1,\ldots,j$ on the first row. Then $\delta^{(j+1)}(T)=(1,0,\ldots,0,1,0)$, 
so the entry $(j+1)$ can be dropped in the first column or the $(j+1)$-th column. Now we have
\begin{equation*}
c_T(P_2+K_j)=\begin{cases}[j]_q!\cdot q\frac{[j-1]_q}{[j]_q}&\text{ if }b_{j+1}(T)=1,\\
[j]_q!\cdot[j+1]_q\frac 1{[j]_q}&\text{ if }b_{j+1}(T)=j+1,\end{cases}\end{equation*} 
which gives us \eqref{eq:tabp2}.
\end{proof}

\begin{example}
For the bowtie graph $G$, only the tableau $T\in\text{SYT}(G+K_3)$ below has the entry $7$ in column $4$. It has $c_T(G+K_j)=q^2[4]_q[2]_q[2]_q$, so $(T_G(q))_{4,3}=q^2[2]_qe_3=(F_G(q))_{4,3}$.
\begin{equation*}
T=\tableau{4&5&6\\1&2&3&7}
\end{equation*}
\end{example}

We now show that every column of $T_G(q)$ contains finitely many nonzero entries.

\begin{proposition}
\label{prop:tabmatrixzeroes}
Let $G=([n],E)$ be an NUIG. If $i\geq n+j$, or if $j\geq i\geq n\geq 2$, then $\text{SYT}^{(i)}(G+K_j)$ is empty, so $(T_G(q))_{i,j}=0$.
\end{proposition}

\begin{proof}
If $i\geq n+j$, then we cannot have any box in column $i$ because there are only $(n+j-1)$ boxes. If $j\geq i\geq n\geq 2$, then by Proposition \ref{prop:tabfirstrow}, the boxes $1,\ldots,i,\ldots,j$ are in the first row. Now, if the entry $(n+j-1)$ is in column $i$, it must be in a different row. But this would give us at least $i+j\geq n+j$ boxes and there are only $(n+j-1)$ boxes.
\end{proof}

We can similarly calculate $X_G(\bm x;q)$ as a matrix multiplication using the tableau matrix.

\begin{proposition}\label{prop:xtabmatrix}
We have
\begin{equation*}
X_G(\bm x;q)=\vec v(q)T_G(q)\vec w^T.
\end{equation*}
\end{proposition}
\begin{proof}
We have 
\begin{equation*}
X_G(\bm x;q)=\sum_{i\geq 1}e_i\left(\sum_{T\in\text{SYT}^{(i)}(G+K_1)}c_T(G+K_1)e_{\text{shape}'(T)}\right)=\sum_{i\geq 1}[i]_qe_i(T_G(q))_{i,1}=\vec v(q)T_G(q)\vec w^T.
\end{equation*}
\end{proof}

We now prove that the operation of gluing graphs also corresponds to multiplying tableau matrices. Informally, the idea is that a tableau for $G+H+K_j$ is essentially a tableau for $H+K_j$ and a tableau for $G+K_k$, for the right choice of $k$.
\begin{proposition}\label{prop:tabmult}
For NUIGs $G$ and $H$, we have
\begin{equation}
T_{G+H}(q)=T_G(q)T_H(q).
\end{equation}
\end{proposition}

\begin{proof}
We will find a bijection
\begin{equation*}
\text{break}:\text{SYT}^{(i)}(G+H+K_j)\to\bigsqcup_{k\geq 1}\text{SYT}^{(i)}(G+K_k)\times\text{SYT}^{(k)}(H+K_j)
\end{equation*}
such that if $\text{break}(T)=(T\vert_G,T\vert_H)\in\text{SYT}^{(i)}(G+K_k)\times\text{SYT}^{(k)}(H+K_j)$, then
\begin{equation}\label{eq:breaktabprops}
c_T(G+H+K_j)=\frac{c_{T\vert_G}(G+K_k)c_{T\vert_H}(H+K_j)}{[k]_q!}\text{ and }
e_{\text{shape}'(T)}=e_{\text{shape}'(T\vert_G)}e_{\text{shape}'(T\vert_H)}.
\end{equation}
Then we will have, as desired, that
\begin{align*}
(T_{G+H}(q))_{i,j}=\frac 1{[i]_q[j-1]_q!}&\sum_{T\in\text{SYT}^{(i)}(G+H+K_j)}c_T(G+H+K_j)e_{\text{shape}'(T)}\\=\sum_{k\geq 1}&\left(\frac 1{[i]_q[k-1]_q!}\sum_{T\vert_G\in\text{SYT}^{(i)}(G+K_k)}c_{T\vert_G}(G+K_k)e_{\text{shape}'(T\vert_G)}\right)\\&\left(\frac 1{[k]_q[j-1]_q!}\sum_{T\vert_H\in\text{SYT}^{(k)}(H+K_j)}c_{T\vert_H}(H+K_j)e_{\text{shape}'(T\vert_H)}\right)\\=\sum_{k\geq 1}&(T_G(q))_{i,k}(T_H(q))_{k,j}=(T_G(q)T_H(q))_{i,j}.
\end{align*}
Let $T\in\text{SYT}^{(i)}(G+H+P_j)$ and let $b=b(T)$. Let $n=|G|$, $n'=|H|$, and let $k=b_{n'+j-1}(T)$ be the column of $T$ that contains the entry $(n'+j-1)$. We now define the tableaux $T\vert_G\in\text{SYT}^{(i)}(G+K_k)$ and $T\vert_H\in\text{SYT}^{(k)}(H+K_j)$ by setting
\begin{equation*}
b(T\vert_G)=(1,2,\ldots,k,b_{n'+j}(T),\ldots,b_{n+n'+j-2}(T))\text{ and }b(T\vert_H)=(b_1(T),\ldots,b_{n'+j-1}(T)).
\end{equation*}
An example is given in Figure \ref{fig:tabbreakexample}.\\

\begin{figure}
$$\begin{tikzpicture}

\node at (5.25,1.5) (){$K_3$};
\node at (2,1.5) (){$G_{}$};
\node at (4,1.5) (){$H_{}$};
\node at (10.5,1.5) (){$G_{}$};
\node at (11.75,1.5) (){$K_3$};
\node at (14.5,1.5) (){$H_{}$};
\node at (15.75,1.5) (){$K_3$};
\draw (1,0)--(1.5,-0.866) (1,0)--(1.5,0.866) (1.5,0.866)--(2.5,0.866) (1,0)--(2.5,0.866)--(2.5,-0.866)--(1,0) (1.5,0.866)--(1.5,-0.866)--(3,0)--(1.5,0.866) (1.5,0.866)--(2.5,-0.866) (1.5,-0.866)--(2.5,0.866) (3,0)--(4,0)--(4.5,0.866)--(5,0)--(4,0);
\draw (2.5,0.866)--(3,0)--(2.5,-0.866)--(1.5,-0.866) (3,0)--(3.5,0.866)--(4.5,0.866) (3.5,0.866)--(4,0);
\draw (5,0)--(5.5,0)--(5.25,0.433)--(5,0);
\filldraw (1,0) circle (3pt) node[align=center,below] (1){1};
\filldraw (1.5,0.866) circle (3pt) node[align=center,above] (2){2};
\filldraw (1.5,-0.866) circle (3pt) node[align=center,below] (3){3};
\filldraw (2.5,0.866) circle (3pt) node[align=center,above] (4){4};
\filldraw (2.5,-0.866) circle (3pt) node[align=center,below] (5){5};
\filldraw (3,0) circle (3pt) node[align=center,below] (6){6};
\filldraw (3.5,0.866) circle (3pt) node[align=center,above] (7){7};
\filldraw (4,0) circle (3pt) node[align=center,below] (8){8};
\filldraw (4.5,0.866) circle (3pt) node[align=center,above] (9){9};
\filldraw (5,0) circle (3pt) node[align=center,below] (10){10};
\filldraw (5.5,0) circle (3pt) node[align=center,below] (12){12};
\filldraw (5.25,0.433) circle (3pt) node[align=center,above] (11){11};

\draw [arrows = {-Stealth[scale=2]}] (6.5,0)--(9,0);
\draw [arrows = {-Stealth[scale=2]}] (9,0)--(6.5,0);

\draw (9.5,0)--(10,0.866)--(11,0.866)--(11.5,0)--(11,-0.866)--(10,-0.866)--(9.5,0) (10,0.866)--(10,-0.866)--(11.5,0)--(10,0.866) (9.5,0)--(11,0.866)--(11,-0.866)--(9.5,0) (10,0.866)--(11,-0.866) (10,-0.866)--(11,0.866) (11.5,0)--(11.75,0.433)--(12,0)--(11.5,0);
\draw (13.5,0)--(16,0) (13.5,0)--(14,0.866)--(14.5,0)--(15,0.866)--(15.5,0)--(15.75,0.433)--(16,0) (15,0.866)--(14,0.866);
\filldraw (9.5,0) circle (3pt) node[align=center,below] (1){1};
\filldraw (10,0.866) circle (3pt) node[align=center,above] (2){2};
\filldraw (10,-0.866) circle (3pt) node[align=center,below] (3){3};
\filldraw (11,0.866) circle (3pt) node[align=center,above] (4){4};
\filldraw (11,-0.866) circle (3pt) node[align=center,below] (5){5};
\filldraw (11.5,0) circle (3pt) node[align=center,below] (6){6};
\filldraw (12,0) circle (3pt) node[align=center,below] (8){8};
\filldraw (11.75,0.433) circle (3pt) node[align=center,above] (7){7};

\filldraw (13.5,0) circle (3pt) node[align=center,below] (1){1};
\filldraw (14,0.866) circle (3pt) node[align=center,above] (2){2};
\filldraw (14.5,0) circle (3pt) node[align=center,below] (3){3};
\filldraw (15,0.866) circle (3pt) node[align=center,above] (4){4};
\filldraw (15.5,0) circle (3pt) node[align=center,below] (5){5};
\filldraw (16,0) circle (3pt) node[align=center,below] (7){7};
\filldraw (15.75,0.433) circle (3pt) node[align=center,above] (6){6};

\node at (1.5,-3) (){$T$};
\draw (2,-3.25)--(5,-3.25) (2,-2.75)--(5,-2.75) (2,-2.25)--(4,-2.25) (2,-1.75)--(3,-1.75) (2,-3.25)--(2,-1.75) (2.5,-3.25)--(2.5,-1.75) (3,-3.25)--(3,-1.75) (3.5,-3.25)--(3.5,-2.25) (4,-3.25)--(4,-2.25) (4.5,-3.25)--(4.5,-2.75) (5,-3.25)--(5,-2.75);
\node at (2.25,-3) (){\textcolor{cyan}1};
\node at (2.75,-3) (){\textcolor{cyan}2};
\node at (3.25,-3) (){\textcolor{cyan}3};
\node at (3.75,-3) (){\textcolor{cyan}4};
\node at (4.25,-3) (){\textcolor{magenta}9};
\node at (4.75,-3) (){\textcolor{magenta}{10}};
\node at (2.25,-2.5) (){\textcolor{teal}5};
\node at (2.75,-2.5) (){\textcolor{teal}6};
\node at (3.25,-2.5) (){\textcolor{teal}7};
\node at (3.75,-2.5) (){\textcolor{magenta}8};
\node at (2.25,-2) (){\textcolor{magenta}{11}};
\node at (2.75,-2) (){\textcolor{magenta}{12}};

\node at (8.5,-3) (){$T\vert_G$};
\draw (9,-3.25)--(12,-3.25) (9,-2.75)--(12,-2.75) (9,-2.25)--(10,-2.25) (9,-3.25)--(9,-2.25) (9.5,-3.25)--(9.5,-2.25) (10,-3.25)--(10,-2.25) (10.5,-3.25)--(10.5,-2.75) (11,-3.25)--(11,-2.75) (11.5,-3.25)--(11.5,-2.75) (12,-3.25)--(12,-2.75);
\node at (9.25,-3) (){\textcolor{teal}1};
\node at (9.75,-3) (){\textcolor{teal}2};
\node at (10.25,-3) (){\textcolor{teal}3};
\node at (10.75,-3) (){\textcolor{magenta}4};
\node at (11.25,-3) (){\textcolor{magenta}5};
\node at (11.75,-3) (){\textcolor{magenta}6};
\node at (9.25,-2.5) (){\textcolor{magenta}7};
\node at (9.75,-2.5) (){\textcolor{magenta}8};
\node at (12.5,-3.3) {\huge{,}};

\node at (13.5,-3) (){$T\vert_H$};
\draw (14,-3.25)--(16,-3.25) (14,-2.75)--(16,-2.75) (14,-2.25)--(15.5,-2.25) (14,-3.25)--(14,-2.25) (14.5,-3.25)--(14.5,-2.25) (15,-3.25)--(15,-2.25) (15.5,-3.25)--(15.5,-2.25) (16,-3.25)--(16,-2.75);
\node at (14.25,-3) (){\textcolor{cyan}1};
\node at (14.75,-3) (){\textcolor{cyan}2};
\node at (15.25,-3) (){\textcolor{cyan}3};
\node at (15.75,-3) (){\textcolor{cyan}4};
\node at (14.25,-2.5) (){\textcolor{teal}5};
\node at (14.75,-2.5) (){\textcolor{teal}6};
\node at (15.25,-2.5) (){\textcolor{teal}7};

\node at (3.25,-4) (){$q^{10}\frac{[2]_q^8}{[5]_q}$};

\node at (10.5,-4) (){$q^7\frac{[3]_q[2]_q^5}{[5]_q}$};

\node at (15,-4) (){$q^3[2]_q^4$};
\end{tikzpicture}$$

\begin{align*}
c_T=& &\textcolor{cyan}{[1]_q\cdot[2]_q\cdot[3]_q\cdot\frac{[4]_q}{[3]_q}}& &\cdot \ \textcolor{teal}{q^2\frac{[2]_q}{[4]_q}\cdot q\frac{[2]_q^2}{[3]_q}\cdot[3]_q}& &\cdot& &\textcolor{magenta}{\frac{[4]_q}{[3]_q}\cdot\frac{[5]_q[2]_q}{[4]_q}\cdot\frac{[6]_q[3]_q}{[5]_q}\cdot q^4\frac{[2]_q}{[6]_q}\cdot q^3\frac{[2]_q^2}{[5]_q}}&\\
c_{T\vert_H}=& &\textcolor{cyan}{[1]_q\cdot[2]_q\cdot[3]_q\cdot\frac{[4]_q}{[3]_q}}& &\cdot \ \textcolor{teal}{q^2\frac{[2]_q}{[4]_q}\cdot q\frac{[2]_q^2}{[3]_q}\cdot[3]_q}& && &&\\
c_{T\vert_G}=& && &\hspace{32pt}\textcolor{teal}{[1]_q\cdot [2]_q\cdot[3]_q}& &\cdot& &\textcolor{magenta}{\frac{[4]_q}{[3]_q}\cdot\frac{[5]_q[2]_q}{[4]_q}\cdot\frac{[6]_q[3]_q}{[5]_q}\cdot q^4\frac{[2]_q}{[6]_q}\cdot q^3\frac{[2]_q^2}{[5]_q}}&
\end{align*}
\caption{\label{fig:tabbreakexample} An example of the map $$\text{break}:\text{SYT}^{(2)}(G+H+K_3)\to\bigsqcup_{k\geq 1}\text{SYT}^{(2)}(G+K_k)\times\text{SYT}^{(k)}(H+K_3).$$}
\end{figure}

For every $1\leq t\leq n'+j-1$, we have $b_t(T\vert_H)=b_t(T)$ and $\delta^{(t)}(T\vert_H)=\delta^{(t)}(T)$, which means that $W_t(T\vert_H)=W_t(T)$, $T\vert_H\in\text{SYT}^{(k)}(H+K_j)$, and $c^{(t)}_{T\vert_H}(H+K_j)=c^{(t)}_T(G+H+K_j)$. Now we have \begin{equation*}
\delta^{(n'+j)}(T)=(1,0,\ldots,0,1,0)=\delta^{(k+1)}(T\vert_G),
\end{equation*} 
so we have the same possible choices of columns in which to drop the subsequent boxes in $T$ and $T\vert_G$. In other words, for $1\leq t\leq n-1$, we have $\delta^{(n'+j-1+t)}(T)=\delta^{(k+t)}(T\vert_G)$, $R_{n'+j-1+t}(T)=R_{k+t}(T\vert_G)$, $W_{n'+j-1+t}(T)=W_{k+t}(T\vert_G)$, and $c^{(n'+j-1+t)}(T)=c_{k+t}(T\vert_G)$. Because $T\in\text{SYT}(G+H+P_j)$, we have 
\begin{equation*}
b_{k+t}(T\vert_G)=b_{n'+j-1+t}(T)\in W_{n'+j-1+t}(T)=W_{k+t}(T\vert_G),
\end{equation*} and since $b_{n+k-1}(T\vert_G)=b_{n+n'+j-2}(T)=i$, then $T\vert_G\in\text{SYT}^{(i)}(G+K_k)$. The inverse map is constructed as follows. Given
\begin{equation*} T\vert_G\in\text{SYT}^{(i)}(G+K_k)\text{ and }T\vert_H\in\text{SYT}^{(k)}(H+K_j),\end{equation*} we must have $b(T\vert_G)_t=t$ for $1\leq t\leq k$ by Proposition \ref{prop:tabfirstrow}, so we can recover the tableau $T\in\text{SYT}^{(i)}(G+H+P_j)$ by setting 
\begin{equation*}
b(T)=(b_1(T\vert_H),\ldots,b_{n'+j-1}(T\vert_H),b_{k+1}(T\vert_G),\ldots,b_{n+k-1}(T\vert_G)).
\end{equation*}
Finally, we check that \eqref{eq:breaktabprops} holds. We have
\begin{align*}
c_{T\vert_H}(H+K_j)c_{T\vert_G}(G+K_k)&=\prod_{t=1}^{n'+j-1}c^{(t)}_{T\vert_H}(H+K_j)\prod_{t=1}^kc^{(t)}_{T\vert_G}(G+K_k)\prod_{t=k+1}^{n+k-1}c^{(t)}_{T\vert_G}(G+K_k)\\&=\prod_{t=1}^{n'+j-1}c^{(t)}_T(G+H+K_j)[k]_q!\prod_{t=n'+j}^{n+n'+j-2}c^{(t)}_T(G+H+K_k)\\&=[k]_q!c_T(G+H+K_k).
\end{align*}
We have $e_{\text{shape}(T)}e_k=e_{\text{shape}(T\vert_G)}e_{\text{shape}(T\vert_H)}$ because $T\vert_H$ contains an additional row of length $k$.
Thus we have $e_{\text{shape}'(T)}=e_{\text{shape}'(T\vert_G)}e_{\text{shape}'(T\vert_H)}$ because we remove an instance of $k$. 
\end{proof}

\begin{proposition} \label{prop:tabpn}
For the path $P_n$, we have $T_{P_n}(q)=F_{P_n}(q)$. 
\end{proposition}

\begin{proof}
If $n\leq 2$, then the result follows from Proposition~\ref{prop:tabp1} and Proposition~\ref{prop:tabp2}. For $n\geq 3$, we have by Proposition~\ref{prop:ftmult} and Proposition~\ref{prop:tabmult} that
\begin{equation*}
T_{P_n}(q)=(T_{P_2}(q))^{n-1}=(F_{P_2}(q))^{n-1}=F_{P_n}(q).
\end{equation*}
\end{proof}

We now show that the tableau matrix and the forest triple matrix agree for every NUIG.
\begin{theorem}\label{thm:samemats}
Let $G=([n],E)$ be an NUIG. Then $F_G(q)=T_G(q)$. 
\end{theorem}

\begin{proof} We first show that $F_G(q)$ and $T_G(q)$ have the same first column, or in other words, that $F_G(q)\vec w^T=T_G(q)\vec w^T$. By Proposition \ref{prop:xmatrix} and Proposition \ref{prop:xtabmatrix}, we can calculate the chromatic quasisymmetric function $X_{P_i+G}(\bm x;q)$ for any $i$ by using either the forest triple matrix or the tableau matrix, that is,
\begin{align*}
X_{P_i+G}(\bm x;q)&=\vec v(q)F_{P_i+G}(q)\vec w^T=\vec v(q)F_{P_i}(q)F_G(q)\vec w^T\text{ and }\\X_{P_i+G}(\bm x;q)&=\vec v(q)T_{P_i+G}(q)\vec w^T=\vec v(q)T_{P_i}(q)T_G(q)\vec w^T.
\end{align*}
By Proposition \ref{prop:tabpn}, we have that $F_{P_i}(q)=T_{P_i}(q)$, so let $M$ be the infinite matrix whose $i$-th row is the row vector $\vec v(q)F_{P_i}(q)=\vec v(q)T_{P_i}(q)$. Then we have that \begin{equation*}
M(F_G(q)\vec w^T-T_G(q)\vec w^T)=0.
\end{equation*} By Proposition \ref{prop:ftmatrixzeroes} and Proposition \ref{prop:tabmatrixzeroes}, the first columns of $F_G(q)$ and $T_G(q)$ satisfy \begin{equation*}(F_G(q))_{i,1}=(T_G(q))_{i,1}=0\end{equation*} for $i\geq n+1$, so we can restrict $M$ to an $n$-by-$n$ matrix $M(n)$ and restrict $F_G(q)\vec w^T-T_G(q)\vec w^T$ to a column vector of size $n$. We now show that this restricted matrix $M(n)$ has $\det(M(n))\neq 0$, which will establish our first claim.\\

By Proposition \ref{prop:ftpn}, we have \begin{equation*}
M_{i,j}=\sum_{k\geq 1}[k]_qe_k(F_{P_i}(q))_{k,j}=\sum_{k\geq 1}[k]_qe_k\sum_{\substack{\alpha\vDash i+j-k-1\\\alpha_\ell\geq j}}([\alpha_1]_q-1)\cdots([\alpha_\ell]_q-1)e_{\text{sort}(\alpha)}.
\end{equation*}
In particular, $M_{i,j}$ is a symmetric function of degree $(i+j-1)$, so $M_{n,n}$ is the only entry of $M(n)$ with degree $(2n-1)$. Furthermore, $M_{n,n}$ contains the term $[2n-1]_qe_{2n-1}$, corresponding to the summand where $k=2n-1$ and $\alpha$ is the empty composition. Therefore, we have
\begin{equation*}
\det M(n)=[2n-1]_qe_{2n-1}\det M(n-1)+g(n),
\end{equation*}
where $g(n)$ is a symmetric function in which $e_{2n-1}$ does not appear. Note that $\det M(1)=M_{1,1}=e_1$. By induction, we have $\det M(n-1)\neq 0$, so the factor $e_{2n-1}$ appears with nonzero coefficient in $\det M(n)$, and in particular, $\det M(n)\neq 0$. This proves that for every NUIG $G$, the matrices $F_G(q)$ and $T_G(q)$ have the same first column.\\

Finally, by Proposition \ref{prop:ftmatrixkj}, the $j$-th column of $F_G(q)$ is simply the first column of $F_{G+K_j}(q)$, divided by $[j-1]_q!$. Similarly, by definition, the $j$-th column of $T_G(q)$ is the first column of $T_{G+K_j}(q)$, divided by $[j-1]_q!$. Because $G+K_j$ is an NUIG, we have shown above that $F_{G+K_j}(q)$ and $T_{G+K_j}(q)$ have the same first column. Therefore, 
\begin{equation*}
(F_G(q))_{i,j}=\frac 1{[j-1]_q!}(F_{G+K_j}(q))_{i,1}=\frac 1{[j-1]_q!}(T_{G+K_j}(q))_{i,1}=(T_G(q))_{i,j},
\end{equation*}
so we have $F_G(q)=T_G(q)$. 
\end{proof}

This connection between $F_G(q)$ and $T_G(q)$ immediately gives us new $e$-positivity results. 

\begin{corollary}\label{cor:ftmatrixpos}
If $G$ is an NUIG, then the entries of $F_G$ are all $e$-positive. 
\end{corollary}

\begin{proof}
This follows because the rational functions $c_T(G)$ satisfy $c_T(1)\geq 0$.
\end{proof}

\begin{corollary}\label{cor:uicycle}
If $G$ is a sum of NUIGs and cycles, then $X_G(\bm x)$ is $e$-positive.
\end{corollary}

\begin{proof}
This follows from Proposition \ref{prop:xmatrix}, Proposition \ref{prop:ftcn}, and Corollary \ref{cor:ftmatrixpos}.
\end{proof}

We also gain new insight on Hikita's standard Young tableaux. Recall that the factors $c_T(q)$ are rational functions in $q$ and may not be polynomials. If we sum over all tableaux $T\in\text{SYT}(G)$ of a fixed shape $\lambda$, we must get a polynomial because it is the coefficient of $e_\lambda$. We can now show that certain smaller partial sums must also be polynomials. 

\begin{corollary}\label{cor:pieces}
Let $G$ be an NUIG and fix an integer $1\leq i\leq n$. We have
\begin{equation}\label{eq:pieces}
\sum_{\substack{\mathcal F\in\text{FT}(G)\\\alpha^{(1)}_1=i, \ r_1=1}}\text{sign}(\mathcal F)q^{\text{weight}(\mathcal F)}e_{\text{type}(\mathcal F)}=\frac 1{[i]_q}\sum_{\substack{T\in\text{SYT}(G)\\n\text{ in column }i}}c_T(G)e_{\text{shape}(T)}.
\end{equation} 
\end{corollary}

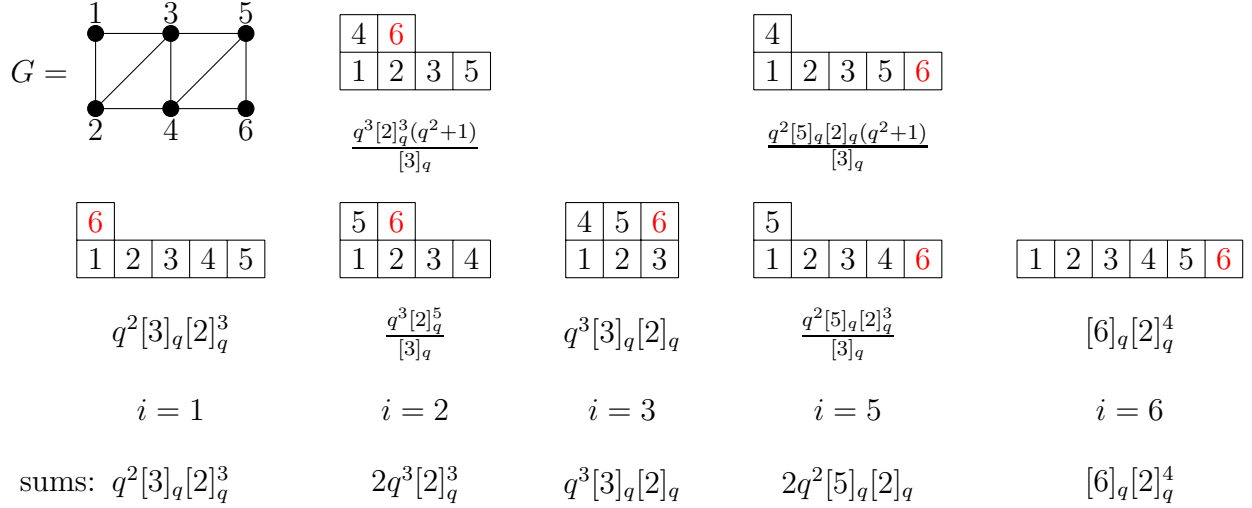
\begin{figure}
$$\begin{tikzpicture}
\draw (-0.75,-0.5) node (1) {$G=$};
\draw (-0.5,-6) node (){sums:};
\filldraw (0,0) circle (3pt) node[align=center,above] (1){1};
\filldraw (0,-1) circle (3pt) node[align=center,below] (2){2};
\filldraw (1,0) circle (3pt) node[align=center,above] (3){3};
\filldraw (1,-1) circle (3pt) node[align=center,below] (4){4};
\filldraw (2,0) circle (3pt) node[align=center,above] (5){5};
\filldraw (2,-1) circle (3pt) node[align=center,below] (6){6};
\draw (0,0)--(2,0)--(2,-1)--(0,-1)--(0,0) (0,-1)--(1,0)--(1,-1)--(2,0);

\draw (0,-2.5) node (){\textcolor{red}6};
\draw (0,-3) node (){1};
\draw (0.5,-3) node (){2};
\draw (1,-3) node (){3};
\draw (1.5,-3) node (){4};
\draw (2,-3) node (){5};
\node at (1,-4) () {$q^2[3]_q[2]_q^3$};
\node at (1,-5) () {$i=1$};
\node at (1,-6) () {$q^2[3]_q[2]_q^3$};
\draw (-0.25,-3.25)--(2.25,-3.25) (-0.25,-2.75)--(2.25,-2.75) (-0.25,-2.25)--(0.25,-2.25) (-0.25,-3.25)--(-0.25,-2.25) (0.25,-3.25)--(0.25,-2.25) (0.75,-3.25)--(0.75,-2.75) (1.25,-3.25)--(1.25,-2.75) (1.75,-3.25)--(1.75,-2.75) (2.25,-3.25)--(2.25,-2.75);

\draw (3.5,0) node (){4};
\draw (4,0) node (){\textcolor{red}6};
\draw (3.5,-0.5) node (){1};
\draw (4,-0.5) node (){2};
\draw (4.5,-0.5) node (){3};
\draw (5,-0.5) node (){5};
\node at (4.25,-1.5) () {$\frac{q^3[2]_q^3(q^2+1)}{[3]_q}$};
\draw (3.25,-0.75)--(5.25,-0.75) (3.25,-0.25)--(5.25,-0.25) (3.25,0.25)--(4.25,0.25) (3.25,-0.75)--(3.25,0.25) (3.75,-0.75)--(3.75,0.25) (4.25,-0.75)--(4.25,0.25) (4.75,-0.75)--(4.75,-0.25) (5.25,-0.75)--(5.25,-0.25);

\draw (3.5,-2.5) node (){5};
\draw (4,-2.5) node (){\textcolor{red}6};
\draw (3.5,-3) node (){1};
\draw (4,-3) node (){2};
\draw (4.5,-3) node (){3};
\draw (5,-3) node (){4};
\draw (3.25,-3.25)--(5.25,-3.25) (3.25,-2.75)--(5.25,-2.75) (3.25,-2.25)--(4.25,-2.25) (3.25,-3.25)--(3.25,-2.25) (3.75,-3.25)--(3.75,-2.25) (4.25,-3.25)--(4.25,-2.25) (4.75,-3.25)--(4.75,-2.75) (5.25,-3.25)--(5.25,-2.75);

\node at (4.25,-4) () {$\frac{q^3[2]_q^5}{[3]_q}$};
\node at (4.25,-5) () {$i=2$};
\node at (4.25,-6) () {$2q^3[2]_q^3$};

\draw (6.5,-2.5) node (){4};
\draw (7,-2.5) node (){5};
\draw (7.5,-2.5) node (){\textcolor{red}6};
\draw (6.5,-3) node (){1};
\draw (7,-3) node (){2};
\draw (7.5,-3) node (){3};
\node at (7,-4) () {$q^3[3]_q[2]_q$};
\node at (7,-5) () {$i=3$};
\node at (7,-6) () {$q^3[3]_q[2]_q$};
\draw (6.25,-3.25)--(7.75,-3.25) (6.25,-2.75)--(7.75,-2.75) (6.25,-2.25)--(7.75,-2.25) (6.25,-3.25)--(6.25,-2.25) (6.75,-3.25)--(6.75,-2.25) (7.25,-3.25)--(7.25,-2.25) (7.75,-3.25)--(7.75,-2.25);

\draw (9,0) node (){4};
\draw (9,-0.5) node (){1};
\draw (9.5,-0.5) node (){2};
\draw (10,-0.5) node (){3};
\draw (10.5,-0.5) node (){5};
\draw (11,-0.5) node (){\textcolor{red}6};
\node at (10,-1.5) () {$\frac{q^2[5]_q[2]_q(q^2+1)}{[3]_q}$};
\draw (8.75,-0.75)--(11.25,-0.75) (8.75,-0.25)--(11.25,-0.25) (8.75,0.25)--(9.25,0.25) (8.75,-0.75)--(8.75,0.25) (9.25,-0.75)--(9.25,0.25) (9.75,-0.75)--(9.75,-0.25) (10.25,-0.75)--(10.25,-0.25) (10.75,-0.75)--(10.75,-0.25) (11.25,-0.75)--(11.25,-0.25);

\draw (9,-2.5) node (){5};
\draw (9,-3) node (){1};
\draw (9.5,-3) node (){2};
\draw (10,-3) node (){3};
\draw (10.5,-3) node (){4};
\draw (11,-3) node (){\textcolor{red}6};
\node at (10,-4) () {$\frac{q^2[5]_q[2]_q^3}{[3]_q}$};
\node at (10,-5) () {$i=5$};
\node at (10,-6) () {$2q^2[5]_q[2]_q$};
\draw (8.75,-3.25)--(11.25,-3.25) (8.75,-2.75)--(11.25,-2.75) (8.75,-2.25)--(9.25,-2.25) (8.75,-3.25)--(8.75,-2.25) (9.25,-3.25)--(9.25,-2.25) (9.75,-3.25)--(9.75,-2.75) (10.25,-3.25)--(10.25,-2.75) (10.75,-3.25)--(10.75,-2.75) (11.25,-3.25)--(11.25,-2.75);

\draw (12.5,-3) node (){1};
\draw (13,-3) node (){2};
\draw (13.5,-3) node (){3};
\draw (14,-3) node (){4};
\draw (14.5,-3) node (){5};
\draw (15,-3) node (){\textcolor{red}6};
\node at (13.75,-4) () {$[6]_q[2]_q^4$};
\node at (13.75,-5) () {$i=6$};
\node at (13.75,-6) () {$[6]_q[2]_q^4$};
\draw (12.25,-3.25)--(15.25,-3.25) (12.25,-2.75)--(15.25,-2.75) (12.25,-3.25)--(12.25,-2.75) (12.75,-3.25)--(12.75,-2.75) (13.25,-3.25)--(13.25,-2.75) (13.75,-3.25)--(13.75,-2.75) (14.25,-3.25)--(14.25,-2.75) (14.75,-3.25)--(14.75,-2.75) (15.25,-3.25)--(15.25,-2.75);
\end{tikzpicture}$$
\caption{\label{fig:tabsums} The tableaux of $\text{SYT}(G)$ organized by the column $i$ of entry $6$.}
\end{figure}

\begin{proof}
The left hand side is $e_i(F_G(q))_{i,1}$ and the right hand side is $e_i(T_G(q))_{i,1}$. 
\end{proof}

\begin{corollary}
\label{cor:polys}
Let $G=([n],E)$ be an NUIG. Then for every partition $\lambda\vdash n$ and integer $1\leq i\leq n$, the sum
\begin{equation*}
c^{(i)}_\lambda(G)=\sum_{\substack{T\in\text{SYT}(G)\\ \text{shape}(T)=\lambda\\n\text{ in column }i}}c_T(G)
\end{equation*}
is a polynomial in the variable $q$ and it is divisible by $[i]_q$.
\end{corollary}

\begin{proof}
Take the coefficient of $e_\lambda$ in \eqref{eq:pieces}.
\end{proof}

\begin{example}
Figure \ref{fig:tabsums} shows an NUIG $G$ and the standard Young tableaux of $\text{SYT}(G)$, organized by the column $i$ in which the largest entry $6$ appears. Although some of the $c_T(G)$ are not polynomials, the sums for each $i$ are polynomials and are divisible by $[i]_q$.
\end{example}

\section{Gluing the first and last vertices}
\label{section:traceresult}

In this section, we show that we can use our matrix $F_G$ to calculate the chromatic symmetric function $X_{G^\circ}(\bm x)$ of the graph $G^\circ$ obtained by gluing vertices $1$ and $n$ of $G$ together.

\begin{definition}
Let $G=([n],E)$ be a graph with $n\geq 2$ vertices. We define the graph
\begin{equation*}
G^\circ=([n-1],\{\{i,j\}\in E: \ 1\leq i<j\leq n-1\}\cup\{\{i,1\}: \ \{i,n\}\in E\}).
\end{equation*}
\end{definition}

\begin{example}
Figure \ref{fig:trace:G0} shows some examples of this graph $G^\circ$. It will be convenient for our proof to allow $G^\circ$ to have multiple edges between the same vertices. Note that multiple edges do not affect the chromatic symmetric function because they do not affect the notion of a proper colouring.
\end{example}
\begin{figure}

$$\begin{tikzpicture}
\node at (-2,-0.866) {$G=$};

    \vertex[red]{1}{(0,0)};
\vertex{2}{(-0.5,0.866)};
\vertex{3}{(-1,0)};
\vertex[black][-3pt][left]{4}{(-0.5,-0.866)};
\vertex[black][-5pt][225]{5}{(-0.5,-1.866)};
\vertex[black][-3pt][270]{6}{(0.5,-1.866)};
\vertex{7}{(0.5,-0.866)};
\vertex[black][-5pt][315]{8}{(1.5,-1.866)};
\vertex[black][-5pt][315]{9}{(1.5,-0.866)};
\vertex[black][-5pt][315]{10}{(2.5,-0.866)};
\vertex[black][-3pt][0]{11}{(3,0)};
\vertex{12}{(2.5,0.866)};
\vertex{13}{(1.5,0.866)};
\vertexl[red][-3.4pt][-3pt]{14}{(1,0)};

\draw (v1)--(v2)--(v3)--(v4)--(v1)--(v3) (v2)--(v4)--(v5)--(v6)--(v7)--(v4)--(v6) (v5)--(v7)--(v8)--(v6) (v7)--(v9)--(v8) (v8)--(v10)--(v9) (v10)--(v11)--(v12)--(v13)--(v14)--(v9)--(v11)--(v13)--(v9)--(v12)--(v14)--(v10)--(v13) (v10)--(v12) (v14)--(v11);

\draw (8.5,0)--(7.5,0.866)--(7,0)--(7.5,-0.866)--(8.5,0) (8.5,0)--(7,0) (7.5,0.866)--(7.5,-0.866)--(8.5,-1.866) (7.5,-0.866)--(9.5,-0.866)--(8.5,0) (7.5,-1.866)--(8.5,-0.866) (8.5,-1.866)--(9.5,-1.866) (7.5,-0.866)--(7.5,-1.866)--(8.5,-1.866)--(8.5,-0.866)--(9.5,-1.866)--(9.5,-0.866)--(10.5,-0.866)--(11,0)--(10.5,0.866)--(9.5,0.866)--(8.5,0) (9.5,-1.866)--(10.5,-0.866) (9.5,-0.866)--(11,0)--(9.5,0.866)--(9.5,-0.866) (10.5,-0.866)--(10.5,0.866)--(8.5,0)--(10.5,-0.866) (10.5,-0.866)--(9.5,0.866) (10.5,0.866)--(9.5,-0.866) (8.5,0)--(11,0);

\vertex[red]{1}[(8,0)]{(0.5,0)};
\vertex{2}[(8,0)]{(-0.5,0.866)};
\vertex{3}[(8,0)]{(-1,0)};
\vertex[black][-3pt][left]{4}[(8,0)]{(-0.5,-0.866)};
\vertex[black][-5pt][225]{5}[(8,0)]{(-0.5,-1.866)};
\vertex[black][-3pt][270]{6}[(8,0)]{(0.5,-1.866)};
\vertex{7}[(8,0)]{(0.5,-0.866)};
\vertex[black][-5pt][315]{8}[(8,0)]{(1.5,-1.866)};
\vertex[black][-5pt][315]{9}[(8,0)]{(1.5,-0.866)};
\vertex[black][-5pt][315]{10}[(8,0)]{(2.5,-0.866)};
\vertex[black][-3pt][0]{11}[(8,0)]{(3,0)};
\vertex{12}[(8,0)]{(2.5,0.866)};
\vertex{13}[(8,0)]{(1.5,0.866)};

\node at (6,-0.866) {$G^\circ=$};

    \tikzmath{\d1 = 1; \h1 = 7; \dy = -3.5; \sqtb2 = {sqrt(3)/2};}
    \node at ({-1-\d1},\dy) {$G=$};
    \vertex[red]{1}[(-1,\dy)]{(0,0.5)};
    \vertex[black][-3pt][270]{2}[(-1,\dy)]{(0,-0.5)};
    \vertex[black][-3pt][90]{3}[(-1,\dy)]{(\sqtb2,0)};
    \vertex[black][-3pt][90]{4}[(-1,\dy)]{(\sqtb2*2,0.5)};
    \vertex[red][-3pt][270]{5}[(-1,\dy)]{(\sqtb2*2,-0.5)};
    \draw (v3)--(v2)--(v1)--(v3)--(v4)--(v5)--(v3);
    \node at (0+\h1-\d1,\dy) {$G^\circ=$};
    \vertex[red]{1}[(\h1,\dy)]{(\sqtb2,0.5)};
    \vertex[black][-3pt][270]{2}[(\h1,\dy)]{(0,0)};
    \vertex[black][-3pt][270]{3}[(\h1,\dy)]{(\sqtb2,-0.5)};
    \vertex[black][-3pt][270]{4}[(\h1,\dy)]{(\sqtb2*2,0)};
    \draw (v3)--(v2)--(v1) (v3)--(v4)--(v1);
    \draw (v3) to[out=60, in=-60] (v1);
    \draw (v3) to[out=120, in=-120] (v1);
\end{tikzpicture}$$
\caption{The graph $G^\circ$ obtained by gluing the first and last vertices of $G$. In the second example, these vertices have a common neighbour in $G$, so we get the multiple edge $\{1,3\}$ in $G^\circ$.}
\label{fig:trace:G0}
\end{figure}
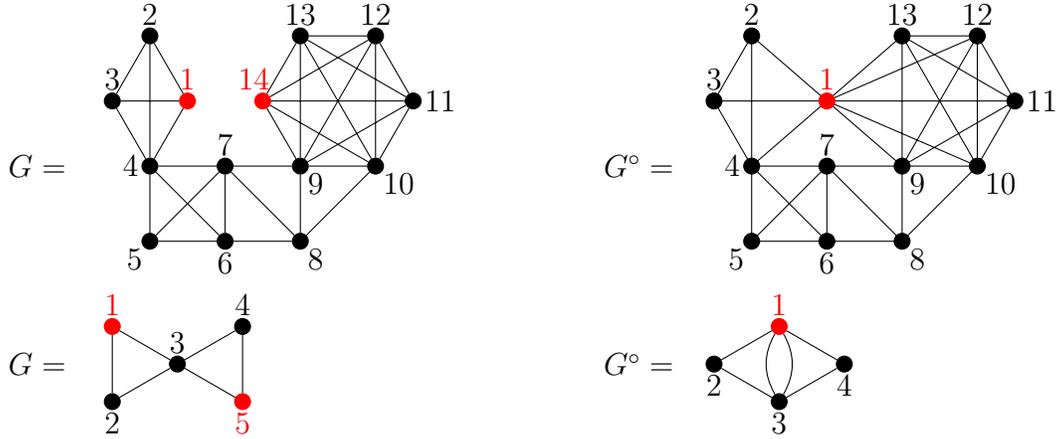

We now state the main result of this section. Note that by Proposition \ref{prop:ftmatrixzeroes}, we have $(F_G)_{k,k}=0$ for $k\geq n$, so it makes sense to take the trace of $F_G$.

\begin{theorem}\label{thm:trace:trace}
Let $G=([n],E)$ be a graph with $n\geq 2$ vertices. Then we have
\begin{equation*}
X_{G^\circ}(\bm x)=\text{trace}(F_G)=\sum_{k=1}^{n-1}(F_G)_{k,k}.
\end{equation*}
\end{theorem}
Before we prove Theorem \ref{thm:trace:trace}, we will discuss some examples and applications.

\begin{example}
For the bowtie graph $G$, we calculated the matrix $F_G$ in Example \ref{ex:bowtiematrix}, so we can calculate the chromatic symmetric function of $G^\circ$, shown in the bottom right of Figure~\ref{fig:trace:G0}, as
\begin{equation*}
X_{G^\circ}(\bm x)=\text{trace}(F_G)=2e_{31}+16e_4.
\end{equation*}
\end{example}

\begin{corollary}\label{cor:circulararcgraph}
If $G$ is an NUIG, then $X_{G^\circ}(\bm x)$ is $e$-positive.
\end{corollary}

\begin{proof}
This follows from Corollary \ref{cor:ftmatrixpos} and Theorem \ref{thm:trace:trace}.
\end{proof}

\begin{remark}
The graph $G^\circ$ arising from an NUIG $G$ in this way will be a \emph{proper circular arc graph}, so Corollary \ref{cor:circulararcgraph} lends support to Conjecture \ref{conj:circui}. 
\end{remark}

%

\begin{figure}
$$\begin{tikzpicture}
\node at (-2.75,0) (){$G^\circ=$};9
\draw (0,0) circle (1.5);
\draw (1.299,0.75) arc[start angle=150, end angle=210, radius=1.5cm];
\draw (-0.75,1.299) arc[start angle=240, end angle = 300,radius=1.5cm];
\draw (1.299,-0.75) arc[start angle=60, end angle=90, radius=5.6cm];
\filldraw (0,1.5) circle (3pt) node[align=center,above] (12){12};
\filldraw [color=red] (0.75,1.299) circle (3pt) node[align=center,above] (1){1};
\filldraw (1.299,0.75) circle (3pt) node[align=center,above] (2){2};
\filldraw (1.5,0) circle (3pt) node[align=center,right] (3){3};
\filldraw (1.299,-0.75) circle (3pt) node[align=center,below] (4){4};
\filldraw (0.75,-1.299) circle (3pt) node[align=center,below] (5){5};
\filldraw (0,-1.5) circle (3pt) node[align=center,below] (6){6};
\filldraw (-0.75,-1.299) circle (3pt) node[align=center,below] (7){7};
\filldraw (-1.299,-0.75) circle (3pt) node[align=center,below] (8){8};
\filldraw (-1.5,0) circle (3pt) node[align=center,left] (9){9};
\filldraw (-1.299,0.75) circle (3pt) node[align=center,above] (10){10};
\filldraw (-0.75,1.299) circle (3pt) node[align=center,above] (11){11};
\node at (-12,0) (){$G=$};
\draw (-11,0)--(-5,0) (-10,0)--(-9.5,0.866)--(-9,0)--(-9.5,-0.866)--(-9,-1.732)--(-8,-1.732)--(-7.5,-0.866)--(-8,0) (-6,0)--(-5.5,0.866)--(-5,0);
\filldraw [color=red] (-11,0) circle (3pt) node[align=center,above] (1){1};
\filldraw (-10,0) circle (3pt) node[align=center,above] (2){2};
\filldraw (-9.5,0.866) circle (3pt) node[align=center,above] (3){3};
\filldraw (-9,0) circle (3pt) node[align=center,above] (4){4};
\filldraw (-9.5,-0.866) circle (3pt) node[align=center,left] (5){5};
\filldraw (-9,-1.732) circle (3pt) node[align=center,below] (6){6};
\filldraw (-8,-1.732) circle (3pt) node[align=center,below] (7){7};
\filldraw (-7.5,-0.866) circle (3pt) node[align=center,right] (8){8};
\filldraw (-8,0) circle (3pt) node[align=center,above] (9){9};
\filldraw (-7,0) circle (3pt) node[align=center,above] (10){10};
\filldraw (-6,0) circle (3pt) node[align=center,above] (11){11};
\filldraw (-5.5,0.866) circle (3pt) node[align=center,above] (12){12};
\filldraw [color=red] (-5,0) circle (3pt) node[align=center,above] (13){13};
\end{tikzpicture}$$
\caption{\label{fig:cyclechord} The graph $(P_2+C_3+P_1+C_6+P_3+C_3)^\circ$ is a cycle $C_{12}$ with three noncrossing chords.}
\end{figure}
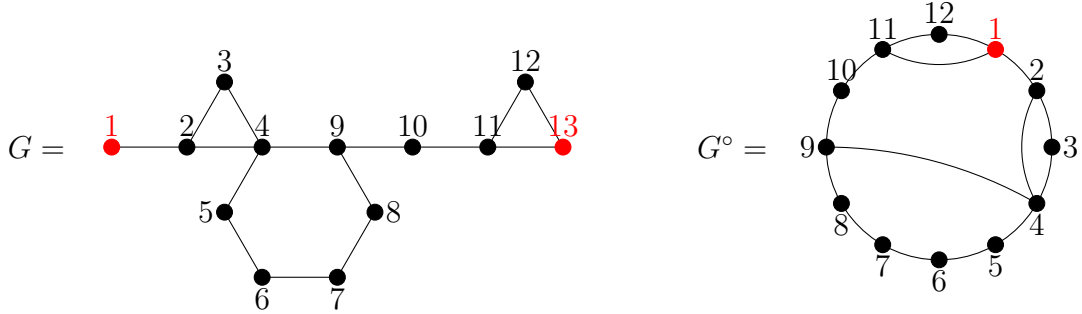

\begin{corollary}\label{cor:circuicycle}
If $G$ is a sum of NUIGs and cycles, then $X_{G^\circ}(\bm x)$ is $e$-positive.
\end{corollary}

\begin{proof}
This follows from Corollary \ref{cor:uicycle} and Theorem \ref{thm:trace:trace}.
\end{proof}

We can also recover some known $e$-positivity results. The \emph{twinned cycle graph} $\tilde C_n$ is the graph obtained by duplicating a vertex of the cycle $C_n$, namely
\begin{equation*}
\tilde C_n=([n+1],\{\{1,2\},\{2,3\}\ldots,\{n-1,n\},\{n,1\},\{n-1,n+1\},\{n,n+1\},\{n+1,1\}\}).
\end{equation*}

\begin{corollary}
\cite[Theorem 3.29]{graphtwin} The twinned cycle graph $\tilde C_n$ is $e$-positive.
\end{corollary}

\begin{proof}
We have $\tilde C_n=(P_{n-2}+K_4')^\circ$, so the result follows from Corollary \ref{cor:circulararcgraph}.
\end{proof}

A \emph{cycle-chord graph} is a graph of the form $C_n\cup\{a,1\}$ obtained by adding some edge $\{a,1\}$ to the cycle $C_n$.

\begin{corollary}
    \cite[Corollary 3.4]{cyclechordsepositive} Cycle-chord graphs are $e$-positive.
\end{corollary}
\begin{proof}
We have $C_n\cup\{a,1\}=(P_a+C_{n-a+2})^\circ$, so the result follows from Corollary \ref{cor:circuicycle}.
\end{proof}

More generally, given $1\leq a_1<b_1\leq a_2<b_2\leq\cdots\leq a_m\leq n$, we can consider the \emph{noncrossing cycle-chord graph} $G$ obtained by adding the edges $\{a_1,b_1\},\ldots\{a_{m-1},b_{m-1}\},\{a_m,1\}$ to $C_n$. These chords can intersect at endpoints but do not cross. An example is given in Figure \ref{fig:cyclechord}. We obtain the following new $e$-positivity result.

\begin{corollary}\label{cor:noncrossingcyclechord}
    Noncrossing cycle-chord graphs are $e$-positive.
\end{corollary}

\begin{proof}
We have
\begin{equation*}
G=(P_{a_1}+C_{b_1-a_1+1}+P_{a_2-b_1+1}+\cdots+P_{a_m-b_{m-1}+1}+C_{n-a_m+2})^\circ,
\end{equation*}
so the result follows from Corollary \ref{cor:circuicycle}. \end{proof}
\begin{remark} Figure~\ref{fig:trace:c6intersectingchord} shows that some cycle graphs with crossing chords are not $e$-positive.
\end{remark}

\begin{figure}
    \begin{tikzpicture}
    \tikzmath{\sqtbt={sqrt(3)/2};};
    \vertex{1}{(0,0)};
    \vertex{2}{(0.5,\sqtbt)};
    \vertex{3}{(1.5,\sqtbt)};
    \vertex{4}{(2,0)};
    \vertex[black][-3pt][270]{5}{(1.5,-\sqtbt)};
    \vertex[black][-3pt][270]{6}{(0.5,-\sqtbt)};
    \draw (v1)--(v2)--(v3)--(v4)--(v5)--(v6)--(v1)--(v4) (v2)--(v5);
    \end{tikzpicture}
    \begin{equation*}
        X_G(\bm x)=2e_{222}\color{red}-6e_{33}\color{black}+26e_{42}+28e_{51}+102e_6
    \end{equation*}
    \caption{\label{fig:trace:c6intersectingchord} The cycle $C_6$ with crossing chords $\{1,4\}$ and $\{2,5\}$ is not $e$-positive.}
\end{figure}

To prove Theorem \ref{thm:trace:trace}, it will be convenient to consider more general objects than forest triples. We will show that they can also be used to calculate the chromatic symmetric function and they will be easier to use in our bijection.
\begin{definition}
A \emph{connected subgraph triple of $G$} is an object $\ctrip=(H,\alpha,r)$ consisting of the following data.
\begin{itemize}
\item $H$ is a connected subgraph of $G$.
\item $\alpha$ is a composition of size $|H|$.
\item $r$ is an integer with $1\leq r\leq\alpha_1$.
\end{itemize}
A \emph{subgraph triple of $G$} is a sequence $\strip=(\ctrip_1=(H_1,\alpha^{(1)},r_1),\ldots,\ctrip_m=(H_m,\alpha^{(m)},r_m))$ of connected subgraph triples such that each vertex of $G$ is in exactly one $H_i$ and we have $\min(H_1)<\cdots<\min(H_m)$. We define the \emph{type} and \emph{sign} of $\strip$ as
\begin{equation*}
\type(\strip)=\sort(\alpha^{(1)}\cdots\alpha^{(m)}), \text{ and }\sign(\strip)=(-1)^{\sum_{i=1}^m\left(\ell(\alpha^{(i)})-1+|H_i|-|E(H_i)|\right)}.
\end{equation*}
We define the \emph{reduced type} of $\strip$, denoted $\type'(\strip)$, to be the partition $\type(\strip)$ with an instance of $\alpha^{(1)}_1$ removed. Let $\ST(G)$ denote the set of subgraph triples of $G$, and as in Definition \ref{def:fmat:fti}, we define the subset
\[\text{ST}^{(i)}(G+P_j)=\{\strip\in\text{ST}(G+P_j): \ \alpha_1=i, r=1, \{n,n+1,\ldots,n+j-1\}\subseteq H', \alpha'_\ell\geq j\},\]
where $\mathcal C=(H,\alpha,r)$ and $\mathcal C'=(H',\alpha',r')$ are the connected subgraph triples of $\mathcal S$ with $1\in H$ and $n\in H'$. Note that we could have $\mathcal C=\mathcal C'$.
\end{definition}

We now show that sums of subgraph triples and forest triples agree.

\begin{proposition}\label{prop:subgraphtriples} Let $G=([n],E)$ be a graph. Then we have \begin{equation}\label{eq:trace:s=f}            \sum_{\strip\in\ST^{(i)}(G+P_j)}\sign(\strip) e_{\type'(\strip)}=\sum_{\ftrip\in\FT^{(i)}(G+P_j)}\sign(\ftrip) e_{\type'(\ftrip)}.\end{equation}
\end{proposition}

    \begin{proof}
        Let $A=\ST^{(i)}(G+P_j)\setminus \FT^{(i)}(G+P_j)$. We will find a bijection $\psi:A\to A$ such that \begin{equation*}\text{type}'(\psi(\mathcal S))=\text{type}'(\mathcal S)\text{ and }\text{sign}(\psi(\mathcal S))=-\text{sign}(\mathcal S).\end{equation*} Then the result will follow because 
\begin{equation*}
\sum_{\mathcal S\in\text{ST}^{(i)}(G+P_j)}\text{sign}(\mathcal S)e_{\text{type}'(\mathcal S)}-\sum_{\mathcal F\in\text{FT}^{(i)}(G+P_j)}\text{sign}(\mathcal F)e_{\text{type}'(\mathcal F)}=\sum_{\mathcal S\in A}\text{sign}(\mathcal S)e_{\text{type}'(\mathcal S)}=0.
\end{equation*} 
Let $\mathcal S=(\mathcal C_i=(H_i,\alpha^{(i)},r_i))\in A$. By definition, some connected subgraph will contain a broken circuit. Let $e$ be the largest edge under our fixed total ordering $\lessdot$ such that some $H_i$ contains a broken circuit of the form $B=C\setminus \{e\}$. We then define the subgraph triple
\begin{equation*}
\psi(\mathcal S)=\begin{cases}
\mathcal S\setminus \mathcal C_i\cup\{(H_i\setminus e,\alpha^{(i)},r_i)\},&\text{ if }e\in H_i,\\
\mathcal S\setminus \mathcal C_i\cup \{(H_i\cup e,\alpha^{(i)},r_i)\},&\text{ if }e\notin H_i.
\end{cases}
\end{equation*}
In other words, we either add or remove the edge $e$ from $\mathcal S$. The map $\psi$ is its own inverse. We have $\text{type}'(\psi(\mathcal S))=\text{type}'(\mathcal S)$ because the compositions have not changed, and we have $\text{sign}(\psi(\mathcal S))=-\text{sign}(\mathcal S)$ because we have changed the number of edges by one. This completes the proof.
\end{proof}

Now we can also calculate the chromatic symmetric function using subgraph triples. 

\begin{corollary}
    We have $X_G(\boldsymbol{x})=\sum_{\strip\in \ST(G)}\sign(\strip) e_{\type(\strip)}$.
\end{corollary}
\begin{proof}
    Using Proposition~\ref{prop:xmatrix} and Proposition~\ref{prop:subgraphtriples}, we get
    \[X_G(\boldsymbol{x})=\sum_{i=1}^n i e_i \left(\sum_{\strip \in \ST^{(i)}(G+P_1)}\sign(\strip) e_{\type'(\strip)}\right)=\sum_{\strip\in\ST(G)}\sign(\strip) e_{\type(\strip)}.\]
\end{proof}

We now prove Theorem~\ref{thm:trace:trace} by finding a bijection on subgraph triples.

\begin{proof}[Proof of Theorem~\ref{thm:trace:trace}]
    Let $A=\bigsqcup_{k}\ST^{(k)}(G+P_k)$. We will find a bijection $\psi:\ST(G^\circ)\rightarrow A$ such that \begin{equation}\label{eq:subgraphbreakprops}\sign(\strip)=\sign(\psi(\strip))\text{ and }\type(\strip)=\type'(\psi(\strip)).\end{equation}
    Then the result will follow because
    \begin{align*}
        X_{G^\circ}(\boldsymbol{x})&=\sum_{\strip\in\ST(G^\circ)}\sign(\strip) e_{\type(\strip)}=\sum_{\strip\in A}\sign(\strip) e_{\type'(\strip)}
        =\sum_{k\geq 1}(F_G)_{k,k}=\trace(F_G).
    \end{align*}
    
Recall that the graphs $G$ and $G^\circ$ have the same number of edges because we are allowing multiple edges in $G^\circ$. Therefore, there is a bijection on the edge multisets $\varphi:E(G)\to E(G^\circ)$ where for $u<v$, we have
\begin{equation*}
\varphi(\{u,v\})=\begin{cases}
\{u,1\},&\text{ if }v=n,\\
\{u,v\},&\text{ otherwise.}
\end{cases}
\end{equation*}
Now for a subgraph $H=(V(H),E(H))$ of $G$, we define the subgraph of $G^\circ$
\begin{equation*}
\varphi(H)=(V(H)\setminus\{n\},\varphi(E(H))),
\end{equation*}
and for a subgraph $H^\circ=(V(H^\circ),E(H^\circ))$ of $G^\circ$, we define the subgraph of $G$
\begin{equation*}
\varphi^{-1}(H^\circ)=(V^*,\varphi^{-1}(E(H^\circ))),\text{ where }V^*=\begin{cases}
V(H^\circ)\cup\{n\},&\text{ if }1\in V(H^\circ),\\
V(H^\circ),&\text{ otherwise.}
\end{cases}
\end{equation*}
    We now define our bijection $\psi$. Some examples are given in Figure \ref{fig:tracepsi}.
    For $\strip\in \ST(G^\circ)$, let $\ctrip^\circ=(H^\circ,\alpha,r)$ be the connected subgraph triple of $\mathcal S$ with $1\in H^\circ$. Note that $\varphi^{-1}(H^\circ)$ has either $1$ or $2$ connected components. First suppose that $\varphi^{-1}(H^\circ)$ contains 2 connected components $H$ and $H'$ with $1\in H$ and $n\in H'$. Because $|\alpha|=|H^\circ|\geq |H'|$, there is some minimal $t$ with $\alpha_1+\cdots+\alpha_t\geq |H'|$; let
    \[k=\alpha_1+\cdots+\alpha_t-|H'|+1.\]
    Then we define $\psi(\strip)$ by replacing $\ctrip^\circ$ with
    \[\ctrip'=\left(H'\cup P_k, \alpha_1\cdots \alpha_t, r\right)\text{ and }\ctrip=\left(H,k \;\alpha_{t+1}\cdots \alpha_{\ell}, 1\right),\]
    where $H'\cup P_k$ is obtained by attaching the path on vertices $n,\ldots,n+k-1$ to $H'$. By minimality of $t$, we have 
 $\alpha_1+\cdots+\alpha_{t-1}<|H'|$ and therefore $\alpha_t\geq k$, so we indeed have $\psi(\mathcal S)\in\text{ST}^{(k)}(G+P_k)$. If instead $\varphi^{-1}(H^\circ)$ has only 1 connected component $H$, then letting $k=r$, we define $\psi(\strip)$ by replacing $\ctrip^\circ$ with
    \[\ctrip=\left(H,k \ \alpha_2\cdots \alpha_\ell \cdot \alpha_1, 1\right).\] Note that $\alpha_1\geq r=k$, so we indeed have $\psi(\mathcal S)\in\text{ST}^{(k)}(G+P_k)$. The map $\psi$ satisfies \eqref{eq:subgraphbreakprops} by construction. We now verify that $\psi$ is a bijection by constructing the inverse. Let $\strip\in A$ and let $\ctrip=(H,\alpha,r)$ and $\ctrip'=(H',\alpha',r')$ be the connected subgraph triples of $\mathcal S$ with $1\in H$ and $n\in H'$. If $\ctrip\neq \ctrip'$, then $\psi^{-1}(\strip)$ is given by replacing $\ctrip$ and $\ctrip'$ with
    \[\ctrip^\circ=\left(\varphi(H\cup H'),\alpha'\cdot \alpha_2\cdots\alpha_\ell, r'\right).\]

    If instead $\ctrip=\ctrip'$, then $\psi^{-1}(\strip)$ is given by replacing $\ctrip$ with
    \[\ctrip^\circ=\left(\varphi(H), \alpha_\ell\cdot\alpha_2\cdots\alpha_{\ell-1},\alpha_1\right).\]
    Note that in this case, $\alpha_1=k\leq\alpha_\ell$, so $\mathcal C^\circ$ is indeed a connected subgraph triple of $G^\circ$. This completes the proof.

    \begin{figure}
        \begin{tikzpicture}
        \tikzmath{\h1=8.5;};
\vertex[red][-3pt][above]{1}[(\h1,0)]{(0,0)};
\vertexs[gray]{2}[(\h1,0)]{(-0.5,0.866)};
\vertexs[teal]{3}[(\h1,0)]{(-1,0)};
\vertexs[teal]{4}[(\h1,0)]{(-0.5,-0.866)};
\vertexs[teal]{5}[(\h1,0)]{(-0.5,-1.866)};
\vertexs[teal]{6}[(\h1,0)]{(0.5,-1.866)};
\vertexs[gray]{7}[(\h1,0)]{(0.5,-0.866)};
\vertexs[teal]{8}[(\h1,0)]{(1.5,-1.866)};
\vertexs[gray]{9}[(\h1,0)]{(1.5,-0.866)};
\vertexs[teal]{10}[(\h1,0)]{(2.5,-0.866)};
\vertexs[gray]{11}[(\h1,0)]{(3,0)};
\vertexs[teal]{12}[(\h1,0)]{(2.5,0.866)};
\vertexs[teal]{13}[(\h1,0)]{(1.5,0.866)};
\vertex[red][-3pt][left]{14}[(\h1,0)]{(1,0)};
\vertex[teal][-3pt][left]{15}[(\h1,0)]{(1,0.5)};
\vertex[teal][-3pt][left]{16}[(\h1,0)]{(1,1)};
\vertex[teal][-3pt][left]{17}[(\h1,0)]{(1,1.5)};


\draw[color=teal] (v1)--(v4)--(v3) (v4)--(v6) (v4)--(v5)--(v6)--(v8)--(v10)--(v13)--(v12)--(v14)--(v15)--(v16)--(v17) (v12)--(v10);

\draw[color=gray] (v7)--(v9)--(v11);

\draw[dashed, opacity=0.4] (v1)--(v2)--(v3)--(v1) (v2)--(v4) (v4)--(v7)--(v6) (v7)--(v5) (v7)--(v8)--(v9)--(v10)--(v11)--(v12)--(v9) (v13)--(v14)--(v9)--(v13) (v14)--(v11);

\node[align=center] at (\h1+1,-2.5) {$\strip\in \ST^{(3)}(G+P_3)$,\\$\textcolor{teal}{\ctrip=\ctrip'=(H,\alpha=3145, r=1)}$};

\tikzmath{\arrplace=(\h1+2)/2;\h1=0;}

\draw [arrows = {-Stealth[scale=2]}] (\arrplace-1.75,0)--(\arrplace+1.75,0);
\draw [arrows = {-Stealth[scale=2]}] (\arrplace+1.75,0)--(\arrplace-1.75,0);

\vertex[red]{1}[(\h1,0)]{(0.5,0)};
\vertexs[gray]{2}[(\h1,0)]{(-0.5,0.866)};
\vertexs[teal]{3}[(\h1,0)]{(-1,0)};
\vertexs[teal]{4}[(\h1,0)]{(-0.5,-0.866)};
\vertexs[teal]{5}[(\h1,0)]{(-0.5,-1.866)};
\vertexs[teal]{6}[(\h1,0)]{(0.5,-1.866)};
\vertexs[gray]{7}[(\h1,0)]{(0.5,-0.866)};
\vertexs[teal]{8}[(\h1,0)]{(1.5,-1.866)};
\vertexs[gray]{9}[(\h1,0)]{(1.5,-0.866)};
\vertexs[teal]{10}[(\h1,0)]{(2.5,-0.866)};
\vertexs[gray]{11}[(\h1,0)]{(3,0)};
\vertexs[teal]{12}[(\h1,0)]{(2.5,0.866)};
\vertexs[teal]{13}[(\h1,0)]{(1.5,0.866)};

\draw[color=teal] (v1)--(v4)--(v3) (v4)--(v6) (v4)--(v5)--(v6)--(v8)--(v10)--(v13)--(v12)--(v1) (v12)--(v10);

\draw[color=gray] (v7)--(v9)--(v11);

\draw[dashed, opacity=0.4] (v1)--(v2)--(v3)--(v1) (v2)--(v4) (v4)--(v7)--(v6) (v7)--(v5) (v7)--(v8)--(v9)--(v10)--(v11)--(v12)--(v9) (v13)--(v1)--(v9)--(v13) (v1)--(v11);

\node[align=center] at (\h1+1,-2.5) {$\strip\in \ST(G^\circ)$,\\$\textcolor{teal}{\ctrip^\circ=(H^\circ=\varphi(H),\alpha=514, r=3)}$};

\tikzmath{\v1=6; \h1=8.5;};
\vertex[red][-3pt][above]{1}[(\h1,\v1)]{(0,0)};
\vertexs[gray]{2}[(\h1,\v1)]{(-0.5,0.866)};
\vertexs[cyan]{3}[(\h1,\v1)]{(-1,0)};
\vertexs[cyan]{4}[(\h1,\v1)]{(-0.5,-0.866)};
\vertexs[cyan]{5}[(\h1,\v1)]{(-0.5,-1.866)};
\vertexs[cyan]{6}[(\h1,\v1)]{(0.5,-1.866)};
\vertexs[gray]{7}[(\h1,\v1)]{(0.5,-0.866)};
\vertexs[magenta]{8}[(\h1,\v1)]{(1.5,-1.866)};
\vertexs[gray]{9}[(\h1,\v1)]{(1.5,-0.866)};
\vertexs[magenta]{10}[(\h1,\v1)]{(2.5,-0.866)};
\vertexs[gray]{11}[(\h1,\v1)]{(3,0)};
\vertexs[magenta]{12}[(\h1,\v1)]{(2.5,0.866)};
\vertexs[magenta]{13}[(\h1,\v1)]{(1.5,0.866)};
\vertex[red][-3pt][left]{14}[(\h1,\v1)]{(1,0)};
\vertex[magenta][-3pt][left]{15}[(\h1,\v1)]{(1,0.5)};
\vertex[magenta][-3pt][left]{16}[(\h1,\v1)]{(1,1)};
\vertex[magenta][-3pt][left]{17}[(\h1,\v1)]{(1,1.5)};


\draw[color=cyan] (v1)--(v4)--(v3) (v4)--(v6) (v4)--(v5)--(v6);

\draw[color=magenta] (v8)--(v10)--(v13)--(v12)--(v14)--(v15)--(v16)--(v17) (v12)--(v10);

\draw[color=gray] (v7)--(v9)--(v11);

\draw[dashed,opacity=0.4] (v1)--(v2)--(v3)--(v1) (v2)--(v4) (v4)--(v7)--(v6)--(v8) (v7)--(v5) (v7)--(v8)--(v9)--(v10)--(v11)--(v12)--(v9) (v13)--(v14)--(v9)--(v13) (v14)--(v11);

\node[align=center] at (\h1+1,-2.75+\v1) {$\strip\in \ST^{(3)}(G+P_3)$\\$\textcolor{cyan}{\ctrip=(H,\alpha=311, r=1)}$\\$\textcolor{magenta}{\ctrip'=(H',\alpha'=26, r'=2)}$};

\tikzmath{\arrplace=(\h1+2)/2;\h1=0;}

\draw [arrows = {-Stealth[scale=2]}] (\arrplace-1.75,\v1)--(\arrplace+1.75,\v1);
\draw [arrows = {-Stealth[scale=2]}] (\arrplace+1.75,\v1)--(\arrplace-1.75,\v1);

\vertex[red]{1}[(\h1,\v1)]{(0.5,0)};
\vertexs[gray]{2}[(\h1,\v1)]{(-0.5,0.866)};
\vertexs[teal]{3}[(\h1,\v1)]{(-1,0)};
\vertexs[teal]{4}[(\h1,\v1)]{(-0.5,-0.866)};
\vertexs[teal]{5}[(\h1,\v1)]{(-0.5,-1.866)};
\vertexs[teal]{6}[(\h1,\v1)]{(0.5,-1.866)};
\vertexs[gray]{7}[(\h1,\v1)]{(0.5,-0.866)};
\vertexs[teal]{8}[(\h1,\v1)]{(1.5,-1.866)};
\vertexs[gray]{9}[(\h1,\v1)]{(1.5,-0.866)};
\vertexs[teal]{10}[(\h1,\v1)]{(2.5,-0.866)};
\vertexs[gray]{11}[(\h1,\v1)]{(3,0)};
\vertexs[teal]{12}[(\h1,\v1)]{(2.5,0.866)};
\vertexs[teal]{13}[(\h1,\v1)]{(1.5,0.866)};

\draw[color=teal] (v1)--(v4)--(v3) (v4)--(v6) (v4)--(v5)--(v6);

\draw[color=teal] (v8)--(v10)--(v13)--(v12)--(v1) (v12)--(v10);

\draw[color=gray] (v7)--(v9)--(v11);

\draw[dashed,opacity=0.4] (v1)--(v2)--(v3)--(v1) (v2)--(v4) (v4)--(v7)--(v6)--(v8) (v7)--(v5) (v7)--(v8)--(v9)--(v10)--(v11)--(v12)--(v9) (v13)--(v1)--(v9)--(v13) (v1)--(v11);

\node[align=center] at (\h1+1,\v1+-2.75) {$\strip\in \ST(G^\circ)$\\$\textcolor{teal}{\ctrip^\circ=(H^\circ,\alpha=2611, r=2)}$\\$\textcolor{teal}{H^\circ=\varphi(H\cup H')}$};
        \end{tikzpicture}
    \caption{\label{fig:tracepsi} Some examples of the map $\psi:\ST(G^\circ)\rightarrow A$.}
    \end{figure}
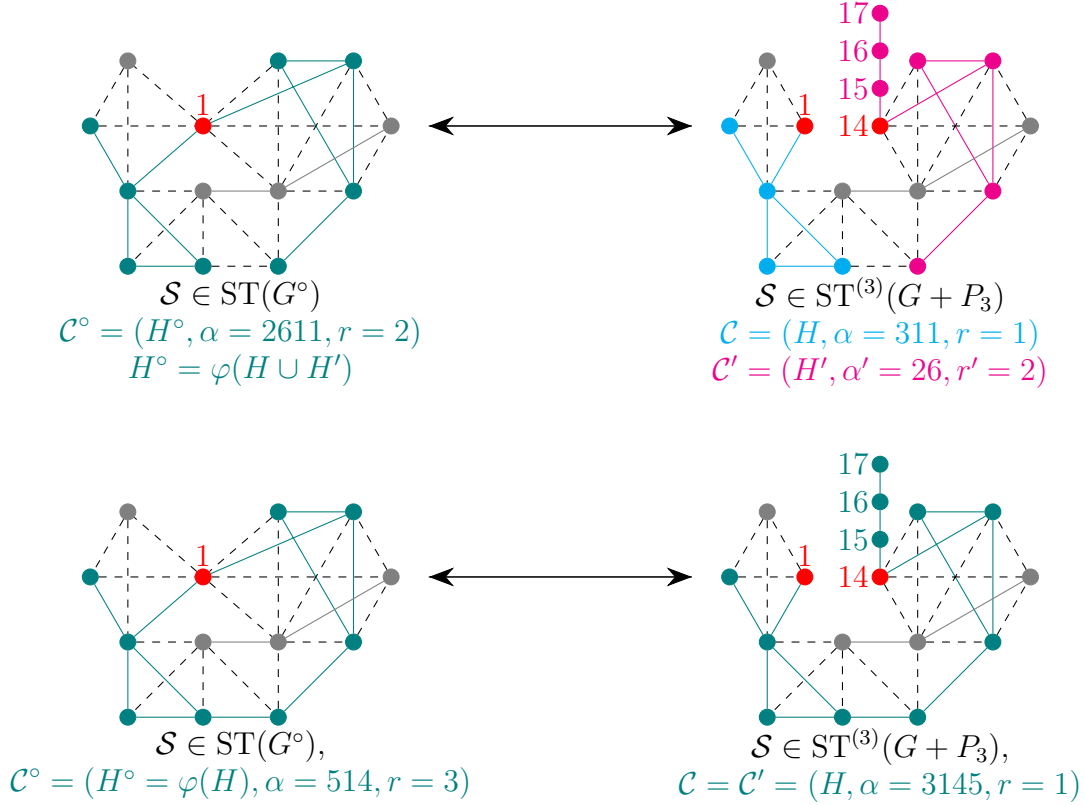

\end{proof}
We can use the formula for $F_{P_n}$ to provide an alternative proof for the chromatic symmetric function of cycles.
\begin{corollary}\cite[Corollary 6.2]{chromquasidi} \label{thm:trace:cg}
The chromatic symmetric function of the cycle $C_n$ is
\begin{equation*}\label{eqn:cyclecsf}X_{C_n}(\boldsymbol{x})=\sum_{\alpha\models n}\alpha_1(\alpha_1-1)\cdots(\alpha_\ell-1)e_{\sort(\alpha)}.\end{equation*}
\end{corollary}
\begin{proof}
    Using Proposition \ref{prop:ftpn} and Theorem \ref{thm:trace:trace}, we have
    \begin{align*}
        X_{C_n}(\boldsymbol{x})&=\trace(F_{P_{n+1}})=\sum_{k\geq 1}\left(\sum_{\alpha\models n, \ \alpha_\ell\geq k}(\alpha_1-1)\cdots (\alpha_\ell-1)e_{\sort(\alpha)}\right)\\
        &=\sum_{\alpha\models n}\alpha_\ell (\alpha_1-1)\cdots(\alpha_\ell-1)e_{\sort(\alpha)}=\sum_{\alpha\models n}\alpha_1(\alpha_1-1)\cdots(\alpha_\ell-1)e_{\text{sort}(\alpha)}.
    \end{align*}
\end{proof}
Curiously, if we include the parameter $q$, we find by Proposition~\ref{prop:ftpn} and \eqref{eq:xqcn} that the trace of $F_{P_{n+1}}(q)$ is
    \[\trace(F_{P_{n+1}}(q))=\sum_{\alpha\models n}\alpha_1([\alpha_1]_q-1)\cdots([\alpha_l]_q-1)e_{\sort(\alpha)}=X_{\vec C_n}(\bm x;q).\]
    This suggests the following $q$-analogue of Theorem~\ref{thm:trace:trace}. Note that we need to carefully keep track of our multiple directed edges in $\vec G^\circ$.
\begin{conjecture}\label{conj:qtrace}
    Let $G=([n],E)$ be an NUIG and let $\vec G^\circ$ be the directed multigraph obtained by directing the edges of $G$ from the smaller vertex to the larger vertex, and then gluing vertices $1$ and $n$. Then we have \[X_{\vec G^\circ}(\bm x;q)=\trace(F_G(q)).\]
\end{conjecture}
\begin{example}
    \begin{figure}
        $$\begin{tikzpicture}[decoration={
    markings,
    mark=at position 0.65 with {\arrow[scale=1.5]{Stealth}}}]
            \tikzmath{\sqtbt={sqrt(3)/2};\w=2;};
            \node at (-0.25,0) {$\Bigg($};
            \node at (3.5,0) {$\Bigg)^\circ\hspace{-0.2cm}=$};
            \vertex{1}{(0,0)};
            \vertex{2}{(1,0)};
            \vertex{3}{(2,0)};
            \vertex{4}{(3,0)};
            \draw (v1)--(v2);
            \draw (v2)--(v3);
            \draw (v3)--(v4);
            \tikzmath{\h1=4.15;\v1=0;};
            \vertex{1}[(\h1,\v1)]{(0,0)};
            \vertex{2}[(\h1,\v1)]{(\sqtbt,0.5)};
            \vertex[black][-3pt][270]{3}[(\h1,\v1)]{(\sqtbt,-0.5)};
            \draw[postaction={decorate}] (v1)--(v2);
            \draw[postaction={decorate}] (v2)--(v3);
            \draw[postaction={decorate}] (v3)--(v1);
    
            \tikzmath{\h1=\h1+\w;\v1=0;};
            \node at (\h1-0.25,0) {$\Bigg($};
            \node at (\h1+2*\sqtbt+0.5,0) {$\Bigg)^\circ\hspace{-0.2cm}=$};
            \vertex{1}[(\h1,\v1)]{(0,0)};
            \vertex{2}[(\h1,\v1)]{(\sqtbt,0.5)};
            \vertex[black][-3pt][270]{3}[(\h1,\v1)]{(\sqtbt,-0.5)};
            \vertex{4}[(\h1,\v1)]{(2*\sqtbt,0)};
            \draw (v1)--(v2);
            \draw (v2)--(v3);
            \draw (v3)--(v4);
            \draw (v1)--(v3);
            \tikzmath{\h1=\h1+2*\sqtbt+1.15;\v1=0;};
            \vertex{1}[(\h1,\v1)]{(0,0)};
            \vertex{2}[(\h1,\v1)]{(\sqtbt,0.5)};
            \vertex[black][-3pt][270]{3}[(\h1,\v1)]{(\sqtbt,-0.5)};
            \draw[postaction={decorate}] (v1)--(v2);
            \draw[postaction={decorate}] (v2)--(v3);
            \draw[postaction={decorate}] (v1)--(v3);
            \draw[postaction={decorate}] (v3) to[out=180, in=-60] (v1);
            \tikzmath{\h1=\h1+\w;\v1=0;};
            \node at (\h1-0.25,0) {$\Bigg($};
            \node at (\h1+2*\sqtbt+0.5,0) {$\Bigg)^\circ\hspace{-0.2cm}=$};
            \vertex{1}[(\h1,\v1)]{(0,0)};
            \vertex{2}[(\h1,\v1)]{(\sqtbt,0.5)};
            \vertex[black][-3pt][270]{3}[(\h1,\v1)]{(\sqtbt,-0.5)};
            \vertex{4}[(\h1,\v1)]{(2*\sqtbt,0)};
            \draw (v1)--(v2);
            \draw (v2)--(v3);
            \draw (v3)--(v4);
            \draw (v1)--(v3);
            \draw (v2)--(v4);
            \tikzmath{\h1=\h1+2*\sqtbt+1.15;\v1=0;};
            \vertex{1}[(\h1,\v1)]{(0,0)};
            \vertex{2}[(\h1,\v1)]{(\sqtbt,0.5)};
            \vertex[black][-3pt][270]{3}[(\h1,\v1)]{(\sqtbt,-0.5)};
            \draw[postaction={decorate}] (v1)--(v2);
            \draw[postaction={decorate}] (v2)--(v3);
            \draw[postaction={decorate}] (v1)--(v3);
            \draw[postaction={decorate}] (v3) to[out=180, in=-60] (v1);
            \draw[postaction={decorate}] (v2) to[out=180, in=60] (v1);
        \end{tikzpicture}$$
        \caption{\label{fig:directedmultigraphs} The directed multigraphs obtained from $P_4$, $K_3+P_2$, and $K_4'$.}
    \end{figure}
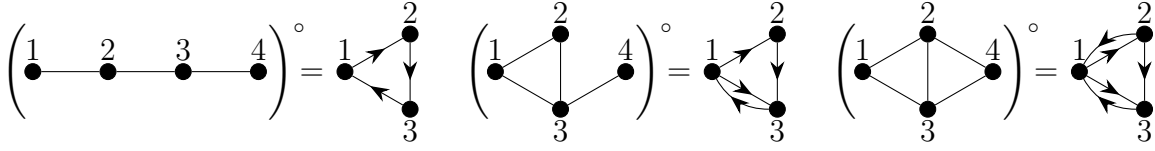
Figure~\ref{fig:directedmultigraphs} shows this construction of $\vec G^\circ$ for the NUIGs $\,G=P_4$, $G=K_3+P_2$, and $G=K_4'$. In each case, the underlying simple graph of $\vec G^\circ$ is $K_3$, and indeed
    \[\trace(F_{P_4})=\trace(F_{K_3+P_2})=\trace(F_{K_4'})=6e_3=X_{K_3}(\bm x).\]
When we calculate the $q$-analogues, multiple edges contribute to the number of ascents, so the $X_{\vec G^\circ}(\bm x;q)$ are no longer the same. It turns out that
    \begin{align*}X_{(\vec P_4)^\circ}(\boldsymbol{x}; q)&=3q[2]_qe_3=\trace(F_{P_4}(q)),\\
    X_{(\overrightarrow{K_3+P_2})^\circ}(\boldsymbol x;q)&=(q^3+4q^2+q)e_3=\trace(F_{K_3+P_2}(q)),\text{ and }\\
    X_{(\vec K_4')^\circ}(\boldsymbol{x};q)&=3q^2[2]_qe_3=\trace(F_{K_4'}(q)),\end{align*}
which is consistent with Conjecture~\ref{conj:qtrace}. We have checked Conjecture~\ref{conj:qtrace} by computer for all NUIGs with at most $8$ vertices.
\end{example}

\printbibliography
\end{document}